\documentclass{article}

\usepackage[utf8]{inputenc}   
\usepackage[T1]{fontenc}      
\usepackage{lipsum}           

\usepackage[a4paper, total={5.5in, 9in}]{geometry}
\usepackage[utf8]{inputenc}
\usepackage{amsmath,amssymb,amsthm, xcolor, enumitem, comment, todonotes}
\usepackage{url}
\usepackage[all]{xy}
\usepackage{hyperref}

\theoremstyle{definition}
\newtheorem{definition}{Definition}

\newtheorem{question}[definition]{Question}
\newtheorem{fact}[definition]{Fact}

\theoremstyle{plain}
\newtheorem{theorem}[definition]{Theorem}
\newtheorem{proposition}[definition]{Proposition}
\newtheorem{lemma}[definition]{Lemma}
\newtheorem{corollary}[definition]{Corollary}

\newtheorem{remark}[definition]{Remark}

\DeclareMathOperator{\fin}{fin}

\newcommand{\fr}{{}^\frown}

\newcommand{\name}{\dot}

\newcommand{\forces}{\Vdash}

\newcommand{\la}{\langle}

\newcommand{\ra}{\rangle}

\newcommand{\uhr}{\restriction}

\newcommand{\power}{\mathcal{P}}

\newcommand{\ol}{\ol}
\newcommand{\po}{\mathbb{P}}

\newcommand{\qo}{\mathbb{Q}}

\DeclareMathOperator{\dom}{dom}

\DeclareMathOperator{\supp}{supp}
\DeclareMathOperator{\suc}{succ}

\newcommand{\ka}{\kappa}

\renewcommand{\power}{\mathcal{P}}

\newcommand{\ptimes}{\ltimes}

\title{On the strength of ultrafilters above choiceless large cardinals and their Prikry forcings}

\author{ 
William Adkisson
\and
Omer Ben Neria\footnote{The second author was partially supported by the Israel Science Foundation (Grant 1302/23).}
}

\date{\today}

\begin{document}

\maketitle

\begin{abstract}
    We study the strength of well-founded ultrafilters on ordinals above choiceless large cardinals and their associated Prikry forcings. Gabriel Goldberg showed that all but boundedly many regular cardinals above a rank Berkeley cardinal carry well-founded uniform ultrafilters. We prove several bounds on the large cardinal strength that is witnessed by such ultrafilters. 
    
    We then extend the theory of Prikry forcing in this context and place limits on the cardinals that can be collapsed or singularized. 
    Finally, we develop the notion of a tensor Prikry system, and use it to give new constructions for several consistency results in choiceless set theory. In particular, we build a new model in which all uncountable cardinals are singular.
\end{abstract}

\section{Introduction}

Large cardinals have proven to be powerful extensions of the axioms of set theory. These large cardinal axioms are often phrased in terms of the existence of elementary embeddings from $V$ to a structure $M$; the strength of the axiom is closely related to the closure of $M$, with embeddings into highly closed models giving rise to very powerful large cardinal properties. However, pushing their strength to its limits conflicts with the axiom of choice: Kunen famously showed that the existence of an elementary embedding from $V$ to itself is incompatible with choice. Therefore, to explore even stronger large cardinal assumptions, one must leave the familiar setting of ZFC and enter the realm of choiceless set theory. While this shift allows for more powerful large cardinals, much of the useful structure provided by ZFC is lost without the axiom of choice.

To overcome this loss, Woodin \cite{Woodin1} has proposed using large cardinal hypotheses such as extendible and supercompact cardinals to derive some of the key consequences of AC. 
This line of research has seen tremendous developments in recent years, following works by Cutolo \cite{Cutolo} proving key choice-like properties at successors of limits of Berkeley cardinals,  Usuba \cite{Usuba1},  who derived introduced the key notion of a Lowenheim-Skolem cardinal as a weak
choice principle, and Aspero \cite{Aspero} who showed how to derive such large cardinals from the existence of a Reinhardt cardinal. 
Schlutzenberg \cite{Schlutzenberg} developed the theory of extenders in these settings and established structural results on elementary embeddings associated with choiceless large cardinals. 
This line of research has notably culminated in the recent work of Goldberg \cite{goldberg:measurablecardinals,goldberg:choicelesscardinals}, who unified and significantly expended these results to derive remarkable choice-related consequences for a proper class of cardinals, and then further develop the theory to shed light on the Continuum Problem. For the most part, the choiceless large cardinal being used to derive these results is a rank Berkeley, whose existence is implied by Reinhardt cardinals but is expressible in first-order set theory. 

A fundamental result in Goldberg's work is the existence of well founded uniform ultrafilters on all sufficiently large regular cardinals. 
    
    \begin{theorem}[Goldberg \cite{goldberg:choicelesscardinals}]\label{THM:UFonregulars}
    Assuming there is a rank Berkeley cardinal $\lambda$ there is some $\kappa \geq \lambda$ (which is almost supercompact) such that for every regular cardinal $\delta  \geq \kappa$ there is a $\kappa$-complete uniform ultrafilter $U$ on $\delta$.
    \end{theorem}

    It is natural to ask about the strength of ultrafilters in these choiceless settings (above a rank Berkeley). In the ZFC setting, ultrafilters are often derived from large cardinals, and the strength of the associated large cardinal determines much of the behavior of these ultrafilters. A normal measure derived from a measurable cardinal behaves very differently than the array of measures that can be derived from strongly compact or supercompact cardinals. 
    In this work, we aim to study the well-founded ultrafilters $U$ on ordinals above a rank Berkeley cardinal through the following two perspectives:
    \begin{itemize}
        \item[(I)] Study properties of $U$ that determine its (large cardinal) strength.

        \item[(II)] Study the theory of Prikry forcing $\po(U)$ with $U$, and iteration of such forcings. 
    \end{itemize}

    There are several ways of making perspective (I) precise. The first is to examine the cardinality of restricted powers. Given an ordinal $\gamma$ and a well-founded ultrafilter $U$ on some set $X$, we can examine $|\gamma^X/U|$. Under choice, suppose $\kappa$ is measurable, with a corresponding normal measure $U$ on $\kappa$, and let $\delta > \gamma$. Then if $\gamma < \kappa$, or if there are $<\delta$ functions from $\gamma$ to $\gamma$, then $|\gamma^\kappa/U| < \delta$. On the other hand, if $U$ is a $\ka$-strongly compact-type ultrafilter on $\delta$, and $\kappa \leq \gamma < \delta$, then $|\gamma^\delta/U| > \delta$.

    Another approach is to use the notion of a $(\kappa,\rho)$ regular ultrafilter.
    An ultrafilter $U$ is $(\kappa,\rho)$ regular for $\ka < \rho$ if it is $\kappa$-complete and has a subset of size $\rho$ all whose intersections indexed by elements of $[\rho]^{\kappa}$ are empty (see Definition \ref{Def:RegularUltrafilters}). In the choice setting, the existence of an ultrafilter $U$ on that is $(\kappa,\rho)$-regular reflects the strength of a $\power_\kappa(\rho)$ strongly compact cardinal.

    Regarding (II), we can ask about the result of forcing with the associated Prikry forcing $\po(U)$. In the ZFC setting, if $U$ is a measure derived from a measurable cardinal, Prikry forcing will not collapse cardinals. On the other hand, if $U$ is derived from a $\kappa$-strongly compact measure on $\delta$ then Prikry forcing with $U$ will collapse all cardinals between $\kappa$ and $\delta$.

    This perspective is closely connected to an independently interesting question: Can choiceless choiceless ultrafilters be used to construct a model where all cardinals are singular?
    Gitik \cite{gitik:AllCardSing} famously constructed this model from the assumption of a class of strongly compact cardinals. 
    Apter \cite{apter:ADsingularcardinals} harnessed the power of AD to give a more direct construction of a model where all cardinals below $\Theta$ are singular, using the result of Steel (see \cite{Steel-HBchapter}) that under AD, all regular cardinals in $L(\mathbb{R})$ below $\Theta$ are measurable.   If $U$ behaves like a simple measurable, then one can try to extend Apter's methods to singularize each $\delta$ without collapsing it. If $U$ behaves more like a strongly compact measure and will collapse many cardinals below $\delta$, then the existence of these ultrafilters would not give us any advantage over Gitik's original construction.\\

    \noindent
    In this work we address these questions. Our main results stem from an analysis of well-founded ultrafilters on ordinals above a rank Berkeley, done in Section \ref{SEC:Choiceless UFs}, whose main finding is:
    \begin{theorem}\label{THM:Main1}
    For every $\gamma$ above a rank Berkeley cardinal $\lambda$, and every well-founded ultrafilter $U$ on a well-ordered set $X$,
    there is a surjection from ${}^\gamma \gamma$ onto $j_U(\gamma) := otp(\gamma^X/U)$.
    \end{theorem}

    \noindent
    The theorem extends Proposition 7.7 in \cite{goldberg:choicelesscardinals}, which uses definability properties to show that for an even ordinal $\epsilon$, if $U$ is a wellfounded ultrafilter on an ordinal $\eta < \theta_{\epsilon+2}$ then $j_U(\eta) = otp(\eta^\eta/U) < \theta_{\epsilon+2}$.\\

    \noindent
    As an immediate Corollary of Theorem \ref{THM:Main1}, we get:
    
    \begin{corollary}
        For every strong limit cardinal $\rho > \lambda$, if $\gamma < \rho$ and $U$ is a well-founded ultrafilter on a well-ordered set $X$ then 
        $|\gamma^X/U| < \rho$. 
    \end{corollary}

    These results provide an answer to the first question we asked: in the sense of restricted powers, ultrafilters above a rank Berkeley must behave more like ultrafilters coming from a measurable, rather than ultrafilters derived from a strongly compact or supercompact cardinal. We can use this to answer our second question:

    \begin{theorem}\label{THM:Main2}
    For all cardinals $\kappa < \rho$ above a rank Berkeley cardinal $\lambda$, if there is no surjection from  $\kappa^\kappa$ onto $\rho$ then there is no $(\kappa,\rho)$-regular ultrafilter on a well-ordered set. 
    \end{theorem}

    In Section \ref{Section:ChoicelessPrikry} we examine the Tree-Prikry forcing $\po(U)$ in the choiceless setting.
    Various forms of Prikry forcing have been developed for use in the choiceless context. Apter \cite{apter:ADsingularcardinals} analyzed the Prikry forcing $\po(U)$ for normal ultrafilters $U$, proving the Prikry property. 
    Also in the context of AD, Steel and Woodin \cite{steel-woodin:HODasCoreModel} developed a version of Tree Prikry forcing for use with the Martin measure.
    They proved the Prikry property for this forcing, and  that infinite subsequences of a Prikry-generic sequence are themselves generic. In \cite{koellner-woodin:LargeCardinalsFromDeterminacy}, Koellner and Woodin discuss a generalized Prikry forcing, and prove that it has the standard Prikry property. A common key point in all of these constructions is that the Prikry property can be verified without Dependent Choice.
    Our focus in Section \ref{Section:ChoicelessPrikry} is on choiceless alternatives to the Strong Prikry Property and the Mathias criterion, with the aim to establish a limitation on the cardinals that could be collapsed in Prikry extensions. 
   
    The largest difficulty in working in the choiceless setting is that much of the standard ultrapower analysis is no longer available. Moreover, we do not know if well-known combinatorial criteria for cardinal preservation such as Rowbottom or Shrinking (see \cite{Devlin}, \cite{Benhamou}) apply to our ultrafilters.
    Nevertheless, by leveraging the previous theorems, we are able to recover some of the useful connections between properties of $U$ and its powers $U^n$, $n < \omega$ and properties of generic extensions $V^{\po(U)}$ by the Prikry forcing $\po(U)$.
    
    In particular, we prove that:
    \begin{theorem}\label{THM:Main3}
        Suppose that $\gamma < \rho$ above a rank Berkeley cardinal such that there is no surjection from $\gamma^\gamma$ to $\rho$ then 
        for any  $\sigma$-complete well-founded ultrafilter $U$ on a well-ordered set,  there is no surjection $f : \gamma \to \rho$ in a generic extension by $\po(U)$.
    \end{theorem}
    As an immediate corollary, we obtain:
    \begin{corollary}\label{COR:PrikryNoCollapseStong}
        If $\rho$ is a strong limit cardinal above a rank Berkeley then it cannot be collapsed in any Prikry generic extension $\po(U)$ by a $\sigma$-complete well-founded ultrafilter $U$ on a well-ordered set.
    \end{corollary}

    We also consider the question of possible changes in cofinalities after forcing with $\po(U)$, which is known to be  tightly connected to the notion of indecomposability. We provide a  criterion for $U$ being $\rho$-indecomposable for a regular cardinal $\rho$ which is related to the existence of suitable rank-Berkeley embeddings $j$, and use it to limit the possible cofinality of $\rho$ in a generic extension by $\po(U)$.

    These results form a natural stepping stone towards making all cardinals singular. Apter's construction in \cite{apter:ADsingularcardinals} obtains this result with a product of Prikry forcings. Because our ultrafilters have relatively low completeness, we are unable to do the same. Moreover, we do not know whether the product of two ultrafilters generates an ultrafilter. In Section \ref{Section:ChoicelessPrikry}, we develop Prikry forcing with tensors of an ultrafilter. This allows us to singularize multiple cardinals simultaneously using a single Prikry forcing on the tensor of different ultrafilters, rather than a product of Prikry forcings on different measures. The key point is since the domain of the tensor is well-ordered, by replacing a product of Prikry forcings with a Prikry forcing on the tensor, we are able to apply our theorem above and conclude that it does not collapse strong limit cardinals. Our final model(s), in which all cardinals are singular, is based on a finite directed system of Prikry forcings using tensors of finite sets of ultrafilters.
    
    Taking a symmetric extension using this forcing, we obtain a model where all cardinals above the rank Berkeley are singular. By adding a final collapse of a desirable initial segment to $\omega$, we will obtain the following theorem:

    \begin{theorem}\label{THM:Main4}
        Assuming ZF plus the existence of a rank Berkeley cardinal $\lambda$, and a strongly inaccessible cardinal $\eta$ above the first almost supercompact cardinal $\kappa_0 \geq \lambda$, there is a symmetric extension of $V_\eta$ satisfying ZF + all uncountable cardinals are singular.
    \end{theorem}

    The forcing construction by which we prove Theorem \ref{THM:Main4} is much closer to Apter's approach then to Gitik's construction, as it builds on adding singularizing sequences to all regular cardinals with the aim to not collapse cardinals. 
    
    In the last section, we review applications and discuss further directions related to these results. Most notably, we give a new construction of a model where the first measurable is the first inaccessible, related to a recent result of Gitik-Hayut-Karagila \cite{gitik-hayut-karagila:FirstMeasurable}.

    \begin{theorem}\label{THM:Main5}
        Assuming ZF plus the existence of a rank Berkeley cardinal $\lambda$, and an inaccessible cardinal $\eta > \lambda$ which is also a limit of almost supercompact cardinals, then there is a symmetric extension of where $\eta$ is both the first inaccessible and the first measurable.
    \end{theorem}

\section{Preliminaries}

We will use the following standard definitions and conventions concerning cardinalities of sets in choiceless settings. For sets $X,Y$, we write $|X| \leq |Y|$ if there is an injection from $X$ to $Y$, and $|X| \leq^* |Y|$ if there is a surjection from $Y$ onto $X$. 
For a set $X$, define $\Theta(X)$ to be the minimal cardinal $\alpha$ such that $|X| \not\geq^* |\alpha|$, and $\aleph(X)$ to be the minimal cardinal $\alpha$ such that $|X| \not\geq |\alpha|$. Note that $\aleph(X) \leq \Theta(X)$.

A cardinal $\kappa$ is strong limit if for every $\gamma < \kappa$, $\Theta(V_{\gamma+1}) < \kappa$, and, as usual, is strongly inaccessible if it is both strong limit and regular.

Rank Berkeley cardinals were introduced by Schlutzenberg. 
\begin{definition}
A cardinal $\lambda$ is rank Berkeley if for all ordinals $\eta > \lambda$ and $\alpha < \lambda$ there is an elementary embedding $j : V_\eta \to V_\eta$ such that $\alpha < cp(j) < \lambda$. 
\end{definition}

\begin{remark}\label{RMK:Fixing-p}
In the definition of a rank Berkeley cardinal, given a finite sequence of ordinals $p \in V_\eta$ we can always find $j : V_\eta \to V_\eta$ such that $j(p) = p$. This can be done by first moving to a larger ordinal $\tau$ that codes both $p,\eta$ (e.g., using normal form or Godel pairing) and taking an elementary embedding $j^*: V_{\tau+1} \to V_{\tau+1}$ such that $\alpha < cp(j^*) < \lambda$.  Then $j^*(\tau) = \tau$ since $\tau$ is definable without parameters in $V_{\tau+1}$. 
And since $\tau$ codes $p,\eta$, $j^*(p) = p$ and $j^*(\eta) = \eta$.  Therefore $j = j^*\uhr V_\eta : V_\eta \to V_\eta$ fixes $p$. 
\end{remark}

We will also make use of the notion of almost supercompact cardinal, introduced by Goldberg in \cite{goldberg:measurablecardinals} and Aspero \cite{Aspero}.
\begin{definition}
A cardinal $\kappa$ is almost supercompact if for every $\eta > \kappa$ and $\alpha < \kappa$ there is some $\bar{\eta} < \kappa$ and an elementary embedding $\pi : V_{\bar{\eta}} \to V_{\eta}$ such that $\pi(\alpha) = \alpha$. 
\end{definition}
A limit of almost supercompact cardinals is almost supercompact. The existence of a rank Berkeley cardinal implies there is a proper class of almost supercompact cardinals (see \cite[Corollary 2.5 and Lemma 2.6]{goldberg:measurablecardinals}).

We follow standard notations and conventions concerning ultrafilters. 
Let $U$ be an ultrafilter on a set $X$. For two functions $f,g$ whose domain is $X$, write $f =_U g$ if $\{ x \in X \mid f(x) = g(x) \} \in U$. The $=_U$ equivalence class of a function $f$ is denoted by $[f]_U$, and for a set $Y$, the set of $=_U$ equivalence classes on ${}^X Y$ is denoted by ${}^X Y/U$.

We say $U$ is well-founded if for every ordinal $\gamma$ the $\in_U$-relation on the set ${}^X \gamma/U$ is a well order.  If $U$ is well-founded we will often identify $\gamma^X/U$ with its ordinal order-type $otp(\gamma^X/U)$ and treat $\gamma^X/U$ as an ordinal.
Assuming DC every $\sigma$-complete ultrafilter is well-founded.  

Without DC, it is shown in \cite{goldberg:measurablecardinals} how almost supercompact cardinals provide alternative to the relevant choice principles. In particular it is shown that:
\begin{theorem}[\cite{goldberg:measurablecardinals}]\label{Fact:AlmostSCWellOrderable}
Suppose that $\kappa$ is an almost supercompact cardinal then
\begin{enumerate}
    \item Every $\kappa$-complete ultrafilter $U$ is well-founded
    \item Every set of $\kappa$-complete ultrafilters on ordinals is well orderable.
\end{enumerate}    
\end{theorem}

We say $U$ is $Z$-closed for a set $Z$ if for every sequence $\la A_z \mid z \in Z\ra \subseteq U$, the intersection $\bigcap_{z \in Z} A_z \in U$. We say $U$ is $R$-complete for a set $R$ if it is $Z$-closed for every $Z \in R$.

\section{Choiceless Ultrafilters on Ordinals}\label{SEC:Choiceless UFs}

Throughout we fix a RB cardinal $\lambda$ in $V$ and an almost supercompact cardinal $\kappa > \lambda$.
Our results concerning the strength of well-founded ultrafilters on ordinals are based on the following notion of weak $U$-equivalence.

\begin{definition}
Let $U$ be an ultrafilter on a set $X$ and $Y$ a class. Define the relation $\sim^{weak}_U$ on the class of functions ${}^X Y$ by 
$f \sim^{weak}_U g$ iff there exists an injective function $t$ such that $g =_U t \circ f$ (I.e., $\{ x \in X \mid g(x) = t(f(x))\} \in U$).\footnote{The notion of a weak equivalence class is related to the notion of a "constellation" from Puritz \cite{Puritz-Constallation}. Thanks to Gabe Goldberg for pointing out this connection.}
\end{definition}

We start by verifying basic properties concerning the weak relation.
\begin{lemma}${}$\label{LEM:WeakEquivBasics}
\begin{enumerate}
\item $\sim^{weak}_U$ is an equivalence relation.
\item  the standard $U$-equivalence (equality mod $U$) refines $\sim^{weak}_U$.
\item If $U$ is well-founded then for every well-orderable set $Y$, the quotient set of equivalence classes $Y^X/\sim^{weak}_U$ is also well-orederable.
\item If $U$ is well-founded then for every well-orderable set $Y$, there is a surjection.
\[
\Phi : \left( Y^X /\sim^{weak}_U \times Y^Y\right) \twoheadrightarrow Y^X/U
\]
\end{enumerate}
\end{lemma}
\begin{proof}${}$
\begin{enumerate}
\item Let $f \sim^{weak}_U g$ be witnessed by an injective function $t$. I.e., $t \circ f =_U g$. Since $t$ is injective, we can compose both sides with $t^{-1}$  and conclude that $f = t^{-1} \circ t \circ f =_U t^{-1} \circ g$ witnessing $g \sim^{weak}_U f$.\\
\noindent
For each function $f \in {}^X Y$, denote its standard $U$-equivalence class by  $[f]_U = \{ g \in Y^X\mid g =_U f\}$ and its weak $U$-equivalence class by $[f]_{\sim^{weak}_U} = \{ g \in Y^X\mid g \sim^{weak}_U f\}$.

\item Clearly, if $f =_U g$ then $f \sim^{weak}_U g$. We can therefore view $\sim^{weak}_U$ as an equivalence relation on the set $Y^X/U$. We denote the induced equivalence relation on $Y^X/U$ by $\sim^{weak}$. Namely, for classes $[f]_U, [g]_U \in Y^X/U$ define $[f]_U \sim^{weak} [g]_U$ if $f \sim^{weak}_U g$. 

\item Let $Y$ be a well-orderable set. The well-foundedness of $U$ implies that the set of  equivalence classes $Y^X/U$ is well-orderable. By the previous clause we may order $Y^X/\sim^{weak}_U$ by associating each $\sim^{weak}_U$ class $C$ with the minimal $Y^X/U$ class contained in it. Denote this class by $C_U$. 

\item For a pair $(C,t) \in Y^X\sim^{weak}_U \times \{ t \in Y^Y \mid t \text{ is injective}\} $, and a function  $f : X \to Y$ representing the minimal $U$-class $C_U$ contained in $C$ (i.e., $C_U= [f]_U$), define $\Phi(C,t) = [t \circ f]_U$. 
This is well-defined since if $f' =_U f$ is another function representing $C_U$ then $t \circ f' =_U t \circ f$.
The definition of $\sim^{weak}_U$ clearly implies $\Phi$ is surjective. 
\end{enumerate}
\end{proof}

\begin{theorem}
Suppose that $U$ is a well-founded ultrafilter on a set $X$ so that for every $j : V_\eta \to V_\eta$, if $X,U \in V_\eta$ and $j(X,U) = X,U$, then $Fix(j) \cap X = \{ x\in X \mid j(x) = x\} \in U$.

For every cardinal $\gamma$, if there is such an embedding $j:V_\eta \to V_\eta$ with $X,U\in V_\eta$, $cp(j) < \lambda$, and $j(X, U) = X, U$, such that $j(\gamma) = \gamma$, then there is some $\bar{\lambda} < \lambda$ and a surjection 
$$\Phi : \bar{\lambda} \times \gamma^\gamma \twoheadrightarrow \gamma^X/U.$$
\end{theorem}

\begin{proof}

Let  $\gamma$ by a cardinal. By Lemma \ref{LEM:WeakEquivBasics} since $U$ is well-founded 
$\gamma^X/\sim^{weak}_U$ is well-orderable. Let 
\[\bar{\lambda} = rank\left(\gamma^X/\sim^{weak}_U \right)\] 
The surjection $\Phi$ from  Lemma \ref{LEM:WeakEquivBasics}, 
maps $\bar{\lambda} \times \gamma^\gamma$ onto ${}^X \gamma /U$. 
It therefore remains to show that $\bar{\lambda} < \lambda$. To this end, let $j : V_\eta \to V_\eta$ be an embedding as in the statement of the Theorem. Namely, $j(X,U,\gamma) = (X,U,\gamma)$. By our assumption, the set
$Fix(j) \cap X = \{ x \in X \mid j(x) = x\} \in U$. 
As $j(X,U,\gamma) = X,U,\gamma$ it is clear that $j$ respects the relation $\sim^{weak}_U$ on classes.
We will finish the proof by showing that $j\uhr \left(\gamma^X/\sim^{weak}_U\right) = id\uhr \left( {\gamma^X/\sim^{weak}_U}\right)$ is the identity, as this would imply that $\bar{\lambda} < cp(j) < \lambda$.

Let $t = (j\uhr \gamma) \in \gamma^\gamma$; note that $t$ is an  injection. We get that for every function $f \in \gamma^X$ and $x \in Fix(j) \cap X$, 
$$
j(f)(x) = j(f)(j(x)) = j(f(x)) = (j\uhr \gamma) \circ f (x)
$$
witnesses that $f \sim^{weak}_U j(f)$. 
\end{proof}

\begin{corollary}\label{COR:gammaModUIsSmall}
Let $\gamma < \rho$ be cardinals above $\lambda$ 
for which $\gamma^\gamma \not \geq^* \rho$.
For every well-founded ultrafilter $U$ on a set $X$ for which there is an elementary embedding
 $j : V_\eta \to V_\eta$ that fixes $X,U,\gamma$ and satisfies $Fix(j) \cap X = \{ x\in X \mid j(x) = x\} \in U$, we have $|\gamma^X/U| < \rho$.
\end{corollary}

We aim to extend the result of the last Corollary to replace the assumptions about $U$ with the assumption that $X$ is well orderable. The following result of Goldberg  will be used to get the desired reduction.

\begin{theorem}[\cite{goldberg:choicelesscardinals}]\label{Prop:Fix(j)InU}
For every well-founded ultrafilter $U$ on an ordinal $\delta$,  and for every large enough $\eta$ with $U \in V_\eta$ there is an elementary embedding $j : V_\eta \to V_\eta$ such that $j(U) = U$ and $Fix(j) \cap \delta \in U$. \\Moreover, given a finite tuple of ordinals $p \in V_\eta$ one can find $j$ that further fixes $p$.
\end{theorem}

We are ready to prove the main result of this section, Theorem \ref{THM:Main1}.
\begin{proof}(Theorem \ref{THM:Main1})\\
We may assume without loss of generality that $U$ is an ultrafilter on a cardinal $\delta$.
    By Theorem
    \ref{Prop:Fix(j)InU},  there is some $\eta$ such that  $j : V_\eta \to V_\eta$ fixes $\delta$, $\gamma$ and $U$, and $Fix(j) \cap dom(U) \in U$. We are therefore done by Corollary \ref{COR:gammaModUIsSmall}.
    \end{proof}

Recall the definition of $(\kappa,\rho)$-regular ultrafilters.
\begin{definition}\label{Def:RegularUltrafilters}
An ultrafilter $U$ is $(\kappa,\rho)$-regular for some $\kappa < \rho$ if it is $\kappa$-complete and there is a sequence $\la A_i \mid i < \rho\ra \subseteq U$ such that for every $I \in [\rho]^{\kappa}$, $\bigcap_{i \in I} A_i = \emptyset$.
\end{definition}

Using Theorem \ref{THM:Main1}, we can now prove 
Theorem \ref{THM:Main2}, showing there are no regular $(\kappa,\rho)$ ultrafilters with a long gap between $\kappa$ and $\rho$.

\begin{proof}(Theorem \ref{THM:Main2})\\
Suppose otherwise. Let $U,X$ give a counter example, and  
$\la A_i \mid i < \rho \ra \subseteq U$ be a witness for $U$ being $(\kappa,\rho)$-regular. 

Define for each $\nu < \rho$ the function $f_\nu : X \to \kappa$ by 
\[f_\nu(x) = otp\left(\{ i < \nu \mid x \in A_i\}\right).\]
For every $\nu < \mu < \rho$ we have  $f_\nu <_U f_\mu$ since for every $x \in A_\nu \in U$, the set
$\{ i < \mu \mid x \in A_i\}$ is an end extension of the set $\{ i < \nu \mid x \in A_i\} \cup \{ \nu\}$, and therefore
\[
f_\mu(x) = otp\left(\{ i < \mu \mid x \in A_i\}\right) \geq otp\left(\{ i < \nu \mid x \in A_i\}\right) + 1  = f_\nu(x)+1.
\]
Therefore,  $\la f_\nu \mid \nu < \rho\ra$ is a witness that $|\kappa^X/U| \geq \rho$, contradicting Theorem \ref{THM:Main1}.
\end{proof}

\section{Tree Prikry Forcing $\po(U)$ in the choiceless setting}\label{Section:ChoicelessPrikry}

In this section, we develop a form of Tree Prikry forcing for use in the choiceless context, and describe a number of its properties. In particular, we prove a weak version of the strong Prikry property, and a Mathias-like characterization of generic sequences. Combining these results with the results of the previous section, we are able to analyze which cardinals will be preserved by forcings of this type.

As an aside, it is worth highlighting one difficulty of working with Prikry forcing without choice. One very useful way to analyze Prikry generics is through their connection with iterated ultrapowers; this connection can fail in the choiceless context, because the ultrapowers in question may not be extensional.

\begin{definition}
Let $U$ be an ultrafilter on a set $X$. The tree Prikry forcing $\po(U)$ consists of pairs $p = (s,T)$ where $s \in X^{<\omega}$ is a finite sequence and $T \subseteq X^{<\omega}$ is a \textbf{$U$-large tree}, meaning that $\la \ra \in T$, and for every $t \in T$, 
$$\suc_T(t) := \{ x \mid t \fr \la x \ra \in T\} \in U.$$

For every $r \in T$ we write
$$T_{r} := \{ t \in X^{<\omega} \mid r \fr t \in T\}.$$

For $p = (s,T)$ we sometimes denote $s = s^p$ and $T = T^p$.

A condition $q$ is a direct extension of $p$, denoted $q \leq^* p$, if $s^q = s^p$ and $T^q \subseteq T^p$. \\
A condition $q$ is a one-point extension of $p$ if there is some $x \in \suc_{T^p}(\la \ra)$ such that 
$s^q = s^p \fr \la x\ra$ and $T^q = T^p_{\la x\ra}$.
In this case, we write $q = p \fr \la x\ra$.\\
A condition $q$ extends $p$, denoted $q \leq p$ if it obtained from $p$ by finitely many one point extensions and direct extensions. 
\end{definition}

\begin{remark}\label{RMK:Homog}
    Similar to the tree Prikry forcing in ZFC, it is clear that the poset $\po(U)$ is homogeneous. Indeed, given conditions $(s,T),(s',T')$ there is a cone isomorphism $\pi$ between the cones over their direct extensions $(s,T \cap T')$, $(s',T \cap T')$, given by $\pi((s \fr t, R)) = (s' \fr t, R)$.
\end{remark}

Before we begin proving properties of this poset, we define measures on $X^n$ that correspond to the set of maximal branches through a $U$-large tree. We follow the presentation from \cite{GoldbergProduct}.

\begin{definition}(Ultrafilter Tensors / Fubini products)
    \begin{enumerate}
        \item Let $U_1,U_2$ be two ultrafilters on sets $X_1,X_2$ respectively. 
        For a set $Z \subseteq X_1 \times X_2$ and $x_1 \in X_1$ let $Z_{x_1} = \{ x_2 \in X_2 \mid (x_1,x_2) \in Z\}$. \\
        Define the (left) tensor product ultrafilter $U_1 \ltimes U_2$ (also known as the Fubini product of $U_1$ and $U_2$) on $X_1 \times X_2$ by 
        $$
        Z \in U_1 \ltimes U_2 \iff \{ x_1 \mid Z_{x_1} \in U_2\} \in U_1.
        $$

        \item For an ultrafilter on a set $U$ we define its finite (left) powers $U^n$ on $X^n$, for $n \geq 1$ by $U^1 = U$ and $U^{n+1} = U^n \ltimes U$.
    \end{enumerate}
\end{definition}

The following basic properties are straightforward consequences of the definition. 
\begin{remark}\label{Remark:ProductBasics}${}$
    \begin{enumerate}
         \item For ultrafilters $U_1,U_2$, the tensor $U_1 \ltimes U_2$ contains the product filter $U_1 \times U_2 = \{ Z_1 \times Z_2 \mid Z_1 \in U_1 \text{ and } Z_2 \in U_2\}$.
         
        \item If $U_1,U_2$ are $\kappa$-complete for some $\kappa$ then so is $U_1\ltimes U_2$.

        \item For an ultrafilter $U$ on a set $X$, we get by a straightforward induction on $n \geq 1$, that a set $Z \subseteq X^n$ belongs to $U^n$ iff there is a tree $S \subseteq T^{\leq n}$ such that 
        \begin{itemize}
            \item $\emptyset \in S$
            \item for every $t \in S$ of length $|t| < n$ the set $\suc_S(t) \in U$
            \item $Z =S \cap \kappa^n$ is the set of maximal branches in $S$.
        \end{itemize}
    \end{enumerate}
\end{remark}

\begin{lemma}\label{Lemma:WSPrikry}(Weak Strong Prikry Property)\\
Suppose 
$U$ is $\sigma$-complete.
For every condition $p \in \po(U)$ and $D \subseteq \po(U)$ which is dense open below $p$, there are $q \leq^* p$ and $n < \omega$ such that for every $t \in T^{q}$ of length $|t| \geq n$ there is a sub-tree $R_t \subseteq T^q_t$ such that $(s^p \fr t, R_t)$ is a direct extension of $q\fr t:= (s^p \fr t, T^q_t)$ and belongs to $D$. Moreover, the direct extension $q$ and the parameter $n$ are determined canonically from $p$ and $D$.
\end{lemma}
\begin{proof}
    Let $p = (s^p,T^p) \in \po(U)$, and let $D\subseteq \po(U)$ be dense open below $p$. We build a sequence of trees $T_k$ recursively, such that either $p\fr t$ has a direct extension in $D$ for all $t \in T_k$ of length $k$, or this holds for no $t \in T_k$ of length $k$.

    For each $n$, define the following function on the set $T_n$ of branches $t\in T^p$ of length $|t| = n$:
    \[f_n(t) = \begin{cases}
        1 & \text{ if } p\fr t \text{ has a direct extension in }D\\
        0 & \text{ otherwise }
    \end{cases}\]
    Since $T$ is a $U$-large tree, by Remark \ref{Remark:ProductBasics}, $T_n \in U^n$, so there must be a $U^n$-large subset $Z_n$ such that $f_n$ is constant on $Z_n$. By the same remark, $Z_n$ is the set of maximal branches in a $U$-large tree $T_n \subseteq (T^p)^{\leq n}$. For each $n$, let $T^*_n$ be the restriction of $T^p$ to nodes that extend $T_n$.
    
    Let $T = \bigcap_{n<\omega} T^*_n$, and let $q = (s^p, T)$. By construction, for each fixed $n$, either all $t \in T$ of length $n$ have a direct extension in $D$ or all $t\in T$ of length $n$ have no direct extensions in $D$.

    First, we claim that there is some $n$ such that for every $t \in T$ of length exactly $n$, $q\fr t$ has a direct extension in $D$. That is, we wish to show that there is some $n$ such that $f_n(t) = 1$ for all $t\in Z_n.$ If not, then for every $n$, there are no $n$-step extensions of $q$ with a direct extension in $D$. Since every extension of $q$ can be decomposed into an $n$-step extension followed by a direct extension, this contradicts the density of $D$.
    
    Now, suppose that every $n$-step extension of $q$ has a direct extension in $D$. We wish to show that every $n+1$-step extension of $q$ also has a direct extension in $D$. Let $t\in T$ with length $|t| = n$. By assumption, for each $t \in T$ of length $n$, there is a subtree $R_t \subseteq T$ with $(s^p\fr t, R_t) \in D$. Since $D$ is dense open, if $x \in \suc_{R_t}(t)$, then $(q \fr  t, R_t) \fr \la x\ra$ must also be in $D$. Since $\suc_{R_t}(t)$ and $\suc_T(t)$ are both in $U$, their intersection will also be in $U$, and in particular will be nonempty. So there is some $x \in \suc_T(t)$ such that $q\fr t \fr \la x \ra$ has a direct extension in $D$. But by construction, since there is one $x$ with this property, every $x$ must have this property.

    To finish the proof, we note that every parameter is definable except the direct extensions $R_t$. Each set $Z_n$ is definable from the function $f$ and the measure $U^n$. The definition of the condition $q = (s^p, T)$ only uses these sets $Z_n$, and we can choose the least $n$ satisfying the statement of the lemma.

\end{proof}

Although weaker than the usual version of the Strong Prikry Property, Lemma \ref{Lemma:WSPrikry} suffices to get the standard Prikry Property. 
\begin{corollary}\label{COR:PrikryProperty}
Let $U$ be a $\sigma$-complete ultrafilter on a set $X$. For every $p \in \po(U)$ and a statement $\sigma$ in the forcing language of $\po(U)$, there is a direct extension $p^*$ of $p$ that decides $\sigma$.
\end{corollary}
\begin{proof}
By the weak Strong Prikry property, for every condition $p = (s,T)$ there is a direct extension $p^* = (s,T^*)$ and $n < \omega$ such that for every $t \in T^*$ of length $|t| \geq n$ there is a $U$-large subtree $T^t \subseteq T^*_t$ such that $(s \fr t, T^t)$ decides $\sigma$. For each such $t$ let $\varepsilon(t) \in 2$ be $0$ if 
$(s \fr t, T^t) \Vdash \sigma$  and $\varepsilon(t) = 1$ if 
$(s \fr t, T^t) \Vdash \neg\sigma$. 
For each $m \geq n$, the $m$-th level of $T^*$,  $T^* \cap X^m$ belongs to the tensor power ultrafilter $U^m$, and therefore there are subsets $B_m \subset T^* \cap X^m$, $B_m \in U^m$ and $\varepsilon_m < 2$ such that $\varepsilon(t) = \varepsilon_m$ for each $t \in B_m$. 
Let $T^{**} \subseteq T^*$ be  sub-tree given by 
\[
T^{**} = \{ t \in T^* \mid \forall m \geq n,len(t) \ \exists t' \in B_m \ t = t'\uhr len(t)\}
\]
$T^{**}$ is a $U$-large tree since $U^m$ is $\sigma$-closed for all $m \geq 1$.

Note that for each $t \in T^{**} \cap X^n$  the $U$-large trees $T^t$ and $T^{**}_t$ intersect, and for every $t'$ in their intersection of  some length $len(t') = m$, the conditions $(s \fr t', T^{t'})$ and $(s \fr t \mid T^t)$ are compatible, and therefore 
$\varepsilon(t') = \varepsilon(t)$. Hence $\varepsilon_m = \varepsilon_n$ for all $m \geq n$.

Let $p^{**} = (s,T^{**})$ be the resulting direct extension of $p^*$.
We claim that $p^{**}$ decides $\sigma$ according to $\varepsilon_n$. 
Say $\varepsilon_n = 0$; we claim that $p^{**} \Vdash \sigma$ (the proof for the case $\varepsilon_n = 1$ and $p^{**} \Vdash \neg \sigma$ is similar). Indeed, let $q = (s\fr t, R)$ be an extension of $p^{**}$ . By extending $q$ several points from $R$ if needed, we may assume that $len(t) \geq n$.
We get that $q$ is compatible with the condition $(s \fr t, T^t)$ which forces $\sigma$. We conclude that the set of extensions of $p^{**}$ that forces $\sigma$ is pre-dense below $p^{**}$. Hence $p^{**} \Vdash \sigma$.
\end{proof}

Recall that an ultrafilter $U$ is $Y$-closed for a set $Y$ if every intersection of sets $\{ A_y \mid y \in Y\} \subseteq U$ belongs to $U$. A standard application of the Prikry property is that if $U$ is sufficiently closed, $\po(U)$ does not add new small sets. In particular, we have:
\begin{corollary}\label{COR:NoNewSmallSets}
    Suppose that $\qo$ is a poset, $Y$ is a set, and $U$ is an ultrafilter that is $\qo \times Y$-closed. Then for every $\qo \times \po(U)$-name $\tau$ for a subset of $Y$, and every condition $p \in \po(U)$ there is a direct extension $p'$ of $p$ and a $\qo$-name $\tau'$ for a subset of $Y$ such that $(1_{\qo},p') \Vdash \tau = \tau'$.
\end{corollary}

Next we prove a theorem that will allow us to connect the results from the previous section to Prikry forcing. 

\begin{theorem}\label{THM:P(U)andU^n}
Suppose that $U$ is a $\sigma$-complete ultrafilter, and $\gamma < \delta$ are cardinals such that $cf(\delta) > \aleph_0$ and $|\gamma^{X^n}/U^n| < \delta$ for all $n < \omega$. Then $\Vdash_{\po(U)}  |\check{\gamma}| < |\check{\delta}|$.
\end{theorem}
\begin{proof}
    Suppose otherwise. Take $p \in \po(U)$ and a $\po(U)$-name $\name{f}$ such that 
    $$
    p \Vdash ``\name{f} : \check{\delta} \to \check{\gamma} \text{ is injective.}"
    $$
    For every $\alpha < \delta$ let 
    $$D_\alpha = \{ q \in \po(U)/p \mid \exists \nu < \gamma \  q \Vdash \name{f}(\check{\alpha}) = \check{\nu}\}.$$
    Each $D_\alpha$ is dense open below $p$. By the weak Strong Prikry Property (Lemma \ref{Lemma:WSPrikry}),
    there is a definable map 
    $$
    \alpha \mapsto p^\alpha,n_\alpha
    $$
    such that for each $\alpha < \delta$, 
    $p^\alpha = (s^p,T^\alpha) \leq^* p$ and $n_\alpha < \omega$ such that for every $t \in T^{\alpha}$ of length $|t| = n_\alpha$ there is a direct extension $(s^p \fr t, R^\alpha_t)$ of $p^\alpha \fr t := (t^p \fr t, T^{p^\alpha}_t)$ and some $\nu^\alpha_t < \gamma$ such that
    $$
    (s^p \fr t, R^{\alpha}_t) \Vdash \name{f}(\check{\alpha}) = \check{\nu}^\alpha_t.
    $$
    Note that since every two direct tree extensions $R^\alpha_t,\bar{R}^\alpha_t$ of $T^{p^\alpha}_t$ are compatible, the decided value $\nu^\alpha_t$ is independent of the specific choice of tree $R^\alpha_t$.
    Hence, the map $\alpha,t \mapsto \nu^\alpha_t$ is definable.

    \noindent
    Since $cf(\delta) > \aleph_0$ there is some $m < \omega$ such that 
    $I = \{ \alpha < \delta \mid n_\alpha = m\}$ has size $|I| = \delta$.
    For each $\alpha \in I$ let $Z_\alpha = T^{p^\alpha} \cap X^m$. As pointed out in Remark \ref{Remark:ProductBasics}, $Z_\alpha \in U^m$. Consider for each $\alpha \in I$ the function $f_\alpha : Z_\alpha \to \gamma$ by 
    $f_\alpha(t)= \nu^\alpha_t$.
    Our assumption that $|\gamma^{X^m}/U^m| < \delta$ implies there are distinct $\alpha \neq \beta$ in $I$ such that $f_\alpha =_U f_\beta$. 
    This means there is some $r \in Z_\alpha \cap Z_\beta \in U^m$ and $\nu < \gamma$ such that 
    $$ \nu^\alpha_r = f_\alpha(r) = \nu =  f_\beta(r) = \nu^\beta_r.$$
    Now, $r \in Z_\alpha \cap Z_\beta = T^{p^\alpha} \cap T^{p^\beta} \cap \kappa^m$ and therefore there are suitable sub-trees $R^\alpha_r,R^\beta_r \subseteq T^{p}_r$ such that the resulting conditions 
    $(s^p \fr r, R^\alpha_r)$ and $(s^p \fr r,R^\beta_r)$ are two direct extensions of $p \fr r$, which force,  
    $\name{f}(\check{\alpha}) = \check{\nu}$ and $\name{f}(\check{\beta}) = \check{\nu}$, respectively. 
    As two direct extensions of a single condition are compatible, we conclude there is an extension $q$ of $p\fr r$ such that 
    $$ q \Vdash \name{f}(\check{\alpha}) = \name{f}(\check{\beta}),$$
    which contradicts our initial assumption  $p \Vdash ``\name{f} \text{ is injective.}"$
\end{proof}

By combining Theorem \ref{THM:P(U)andU^n} with Theorem \ref{THM:Main1} we can prove Theorem \ref{THM:Main3}, showing that, for every $\gamma < \rho$ above a rank Berkeley cardinal with $\gamma^\gamma \not\geq^* \rho$, no Prikry forcing $\po(U)$ with a $\sigma$-complete ultrafiler $U$ on a well orderable set can collapse $\rho$ to $\gamma$.

\begin{proof}(Theorem \ref{THM:Main3})\\
    Let $U$ be a $\sigma$-complete well-founded ultrafilter on a well-orderable set, which we may assume to be a cardinal $\delta$.  By Theorem \ref{THM:P(U)andU^n} above, to show that $\po(U)$ does not collapse $\rho$ to $\gamma$ it suffices to verify that 
    $|\gamma^{\delta^n}/U^n| < \rho$ for each finite $n \geq 1$. For this, note that as $U^n$ is $\sigma$-complete well-founded, and its domain is the well-orderable set $\delta^n$. Hence, by Theorem \ref{THM:Main1}, $|\gamma^{\delta^n}/U^n| < \rho$.
\end{proof}

\subsection{Decomposibility and singularizing properties of $\po(U)$}

Results in Apter and Gitik, have further studied the possible configurations of regular/singular cardinals. 
Theorem \ref{THM:Main3} and Corollary \ref{COR:PrikryNoCollapseStong} places some bounds on the cardinals $\rho$ that can be collapsed in a generic extension by a Prikry forcing $\po(U)$ with a $\sigma$-complete ultrafilter $U$. If $U$ is uniform on a regular cardinal $\delta$ then $\po(U)$ singularizes $\delta$, and it is natural to ask about other cardinals that could be singularized by $\po(U)$.
In the standard ZFC setting, it is well-known that the question is related to the discontinuity points of the ultrapower embedding $j_U : V \to M_U \cong Ult(V,U)$ (e.g., Lemma 10 in \cite{Hayut-BD}). 
We say that $j$ is discontinuous at an ordinal $\alpha$ if $\sup(j''\alpha) < j(\alpha)$.
In our choiceless setting, given a well-founded ultrafilter $U$ on a set $X$, we write $\sup_{\gamma < \rho}(\gamma^X/U) < \rho^X/U$. 

An argument similar to the proof of Theorem \ref{THM:Main3}, using the weak strong-Prikry-property (Lemma \ref{Lemma:WSPrikry}), gives a similar results about singularizing $\rho$ in our choiceless setting.

\begin{lemma}
    Suppose that $U$ is a $\kappa$-complete ultrafilter on a set $X$ and $\rho$ is a regular cardinal so that for every $n < \omega$, 
    $\sup_{\gamma < \rho} (\gamma^{X^n}/U^n) < (\rho^{X^n}/U^n)$. Then $\po(U)$ does not add a cofinal sequence to $\rho$ of length $\tau < \kappa$. 
\end{lemma}

We verify that the continuity of $U$ at $\rho$ carries to finite tensor powers $U^n$ for other $n \geq 1$.
\begin{lemma}
    Let $U$ be a well-founded ultrafilter and $\rho$ a regular cardinal. 
    If $\sup_{\gamma < \rho}(\gamma^X/U) = \rho^X/U$ then
$\sup_{\gamma < \rho}(\gamma^{X^n}/U^n) = \rho^{X^n}/U^n$ for every $1 \leq n < \omega$.
\end{lemma}
\begin{proof}
    By induction on $n \geq 1$. 
    We assume that the assertion holds for $n$, and prove it for $n+1$.
    We need to verify that for every $f : X^{n+1} \to \rho$ there is some $\gamma^* < \rho$ such that $f^{-1}(\gamma^*) \in U^{n+1}$. 
    For each $x \in X$ let $f_x : X^n \to \rho$ given by $f_x(\vec{y}) = f(\la x\ra \fr \vec{y})$. 
    By the inductive assumption there is some $\gamma < \rho$ such that $f_x^{-1}(\gamma) \in U^n$. Denote the minimal such $\gamma$ by $g(x)$. 
    Now, $g : X \to \rho$ so there is $\gamma^* < \rho$ such that $g^{-1}(\gamma^*) \in U$. It follows from the definition of $U^{n+1}$ that 
    $f^{-1}(\gamma^*) \in U^{n+1}$.
    \end{proof}

Combining the last two results, we get
\begin{corollary}\label{COR:P(U)SingRho}
     Suppose that $U$ is a $\sigma$-complete ultrafilter on a set $X$ and $\rho$ is a regular cardinal such that 
    $\sup_{\gamma < \rho} (\gamma^{X}/U) < (\rho^{X}/U)$. Then $\po(U)$ does not add a cofinal sequence to $\rho$ of length $\tau < \kappa = cp(U)$. 
\end{corollary}

The key combinatorial property for $U$ which is connected to its continuity at $\rho$ is the one of decomposability. 

\begin{definition}
    An ultrafilter $U$ is $\rho$-decomposable if there is a sequence $\la Z_\alpha \mid \alpha \in \rho\ra$ of small sets $Z_\alpha \not\in U$ such that $\bigcup_{\alpha < \rho} Z_\alpha \in U$ but  $\bigcup_{\alpha < \beta}Z_\alpha \not\in U$ for all $\beta < \rho$. 
\end{definition}

$U$ is $\rho$-indecomposable if it is not $\rho$-decomposable.
There is a well-known connection between the notions of continuity and decomposability.
\begin{fact}
    $U$ is $\rho$-decomposable iff 
    $\sup_{\gamma < \rho} (\gamma^{X}/U) < (\rho^{X}/U)$
\end{fact}

In \cite{goldberg:measurablecardinals} Goldberg uses embeddings associated with almost supercompact cardinals $\kappa < \delta$ to prove that various filters $F$ on $\delta$ (such as the club filter on $\delta$) are $\rho$-indecomposable for every $\rho < \kappa$, and are therefore $\kappa$-complete. Roughly speaking, the proof relies on the existence of elementary embeddings $\pi$ satisfying the continuity property at $\rho$ of $\pi(\rho) = \rho$, and the discontinuity property $\bigcap \pi`` U \in \pi(U)$ for $U$. 

Here, we would like to extend this idea to show that the $\rho$-decomposability of $U$ follows from any ``mismatch'' in continuity by an elementary embedding $j$. Using the concepts and notions from Goldberg's work \cite{GoldbergProduct}, we can prove a slightly more general statement, involving two filters $F,G$.

\begin{definition}(Kat\'etov order)\\
    Let $F,G$ be filters on sets $X = \cup F$ and $Y = \cup G$ respectively. We write $F \leq_{kat} G$ if there is a function $f : Y \to X$ such that $F \subseteq f_* G := \{ A \subseteq X \mid f^{-1}(A) \in G\}$.
\end{definition}

The following observation from \cite{GoldbergProduct} makes the needed connection between the Kat\'etov order and decomposability. 

\begin{fact}
    An ultrafilter $U$ is $\rho$-decomposable iff $F_\rho(\rho) \leq_{kat} U$, where $F_\rho(\rho)$ is the $\rho$-Frechet (co-bounded) filter on $\rho$.
\end{fact}
 
Adding the last to Corollary \ref{COR:P(U)SingRho}, we conclude 

\begin{corollary}\label{COR:Kat2Sing}
If $\rho$ is a regular cardinal such that  $F_\rho(\rho) \not\leq_{kat} U$ then $\po(U)$ does not add a cofinal sequence to $\rho$ of length $\tau < \kappa = cp(U)$. 
\end{corollary}

Motivated to find conditions under which the assumption of Corollary \ref{COR:Kat2Sing} is secured, we consider different notions by which an elementary embedding $j : V_\eta \to M$ is continuous/discontinuous at a filter $F$. To make the definition nontrivial, we will always assume our filter $F$ does not have principal ultrafilter extensions, i.e., $\bigcap F = \emptyset$.

\begin{definition}${}$
Let $V_\eta$ with $F \in V_\eta$.
\begin{enumerate}
    \item We say that an elementary embedding $j : V_\eta \to M$ is  Type I discontinuous at $F$ if 
    $\bigcap j``F \neq \emptyset$.

    \item We say that $h : V_\eta \to M$ is  Type II discontinuous at $F$ if $\bigcap j``F \in j(F)$.
\end{enumerate}
\end{definition}

Next, we consider possible notions of continuity of an embedding $j$ at a filter $F$. It makes sense that a good notion of continuity of $j$ at $F$ would involve having $j``F$ (nearly) generating $j(F)$. It is not clear if this can be achieved at all without the assumption of $j[\cup F] \in j(F)$. The next Lemma explains why the requirement $j[\cup F] \in j(F)$ makes a sensible notion of continuity.


\begin{lemma}
    If $j$ satisfies $j[\cup F] \in j(F)$ then $j(F)$ is generated by $j``F \cup \{ j[\cup F]\}$. 
\end{lemma}
\begin{proof}
Let $X = \cup F$ the domain of $F$. By assumption $j[X] \in j(F)$.
    For every $A^* \in j(F)$, let $A = j^{-1}(A^*) =j^{-1}(A^* \cap j[X])$.
    $A$ must be in $F$ as $j(A) \supseteq A^* \cap j[X]$. 
    Also $j(A) \cap j[X] = A^* \cap j[X]\subseteq A^*$.
\end{proof}

\begin{definition}${}$
\begin{enumerate}
    \item We say $j$ is Type I continuous at $F$ if $ \bigcap j`` F  = \emptyset$.
    
    \item We say $j$ is Type II continuous at $F$ if $j[\cup F] \in j(F)$.
\end{enumerate}
 
 \end{definition}

 \begin{remark}\label{RMK:F_rho(rho)}
     Note that for a regular cardinal $\rho \in V_\eta$ and an elementary embedding $j :V_\eta \to M$, 
     $j$ is continuous at $\rho$ iff if is type I continuous at $F_\rho(\rho)$ iff it is type II continuous at $F_\rho(\rho)$. And similarly, it is discontinuous at at $\rho$ iff it is type I discontinuous at $F_\rho(\rho)$ iff it is  type II discontinuous at $F_\rho(\rho)$.
 \end{remark}

\begin{theorem}
 Suppose that $F,G$ are filters without principal ultrafilter extensions, and $j : V_\eta \to M$ is an elementary embedding such that $F,G \in V_\eta$. 
 \begin{enumerate}
     \item If $j$ is type I continuous at $F$ and  type I discontinuous at $G$ then $F \not\leq_{kat} G$.

     \item If $j$ is  type II continuous at $F$ and  type II discontinuous at $G$ then $G \not\leq_{kat} F$.
 \end{enumerate}
\end{theorem}

\begin{proof}${}$
Denote $X = \cup F$, $Y = \cup G$.
    \begin{enumerate}
        \item Suppose otherwise, let $f : Y \to X$ such that $F \subseteq f_* G$. 
Let $\beta \in \bigcap j``G \in j(G)$. 
For every $A \in F$, $f^{-1}(A) \in G$ so $\beta \in  j(f^{-1}(A)) = j(f)^{-1}(j(A))$. Therefore, $\alpha= j(f)(\beta) \in j(A)$. It follows that 
$\alpha \in \bigcap j``F$, contradicting the assumption 
that $j$ is type I continuous at $F$.

\item 
Suppose otherwise. Let $g : X \to Y$ such that $G \subseteq g_* F$,
 $B^* = \bigcap j`` G \in j(G)$, and $A^* = j(g)^{-1}(B^*) \in j(F)$. 
Since $j$ is type II continuous at $F$ there is some $A \in F$ such that $j(A) \cap j[X] \subseteq A^*$. 
It follows that for every $B \in G$, 
$j(A) \cap j[X] \subseteq j(g^{-1}(B))$,  which implies that $A \subseteq g^{-1}(B)$, and thus, that $g``A \subseteq B$. 
This contradicts the assumption that $G$ does not have principal ultrafilter extensions. 
    \end{enumerate}
\end{proof}

By combining the last Theorem with Remark \ref{RMK:F_rho(rho)} and Corollary \ref{COR:Kat2Sing} we get
\begin{corollary}\label{COR:DistinguishingEmbedding}
    Let $U$ be a $\sigma$-complete ultrafilter on a set $X$ and $\rho$ a regular.
    If there is $V_\eta$ with $\rho,U \in V_\eta$ and elementary embedding $j : V_\eta \to M$ that is either
    \begin{itemize}
        \item continuous at $\rho$ and type I discontinuous at $U$, or 
        \item discontinuous at $\rho$ and type II continuous at $U$,
    \end{itemize}
    then $U$ is $\rho$-indecomposable and $\po(U)$ does not add a cofinal sequence to $\rho$ of length $\tau < \kappa = cp(U)$. 
    \end{corollary}

\subsection{A choiceless version of the Mathias criterion}\label{Section:Mathias}
We wish to prove an analogue of the Mathias Criterion for this poset. The details are more technical than the standard Mathias Criterion for Prikry forcing, as a consequence of only having access to the weak form of the standard strong Prikry lemma. We will phrase this criterion in terms of the following systems:

\begin{definition}
    A $U$-Subtree-Predicate-System ($U$-SPS)  consists of a triple $(T,n,P)$ such that 
    \begin{itemize}
        \item $T$ is a $U$-large tree
        \item $n < \omega$
        \item $P$ is a predicate on pairs $(t,S)$ where $t$ is a finite sequence with length at least $n$ and $S$ is a $U$-large sub-tree of $T_t$.
        \item for every $t \in T$ of length $|t| \geq n$ the set $\{ S \mid P(t,S) \}$ is nonempty and downwards closed.
    \end{itemize}

    We say that a sequence $\vec{x} = \la x_i \mid i < \omega\ra$ \emph{fits} $(T,n,P)$ if $\vec{x}$ is a branch through $T$ and there is some $k \geq n$ and a $U$-large sub-tree $S$
   of $T_{\vec{x}\uhr k}$ so that $P(\vec{x}\uhr k,S)$ holds and $\vec{x}\setminus k$ is a branch through $S$.
\end{definition}

\begin{definition}${}$
\begin{enumerate}
    \item
    A Global Predicate System (GPS) for $U$, is an assignment $\Phi$ whose domain is $X^{<\omega}$ and for every $r \in X^{<\omega}$, 
    $\Phi(r) = (T^r,n^r,P^r)$ is a $U$-SPS.

    \item The $\po(U)$-filter generated by an infinite sequnece $\vec{x} \in X^\omega$ is given by 
    \[
    G(\vec{x}) = \{ (\vec{x}\uhr k , T) \mid k < \omega, T \text{ is } U\text{-large, and } \vec{x}\setminus k \in [T]\}.
    \]

    \item We say that a sequence $\vec{x} \in X^\omega$ is $\po(U)$-generic when the filter  $G(\vec{x})$ is generic.
    \end{enumerate}
\end{definition}

\begin{theorem}\label{THM:AlternativeMathias}(A choiceless Mathias Criterion)\\
    Let $U$ be a $\sigma$-complete measure on a set $X$. A sequence $\vec{x} = \langle x_i \mid i<\omega\rangle \in X^\omega$ generates a generic for $\po(U)$ iff 
    for every Global Predicate System $\Phi$ for $U$ there is some $m < \omega$ such that
    $\vec{x}\setminus m$ fits the $U$-SPS $\Phi(\vec{x}\uhr m)$.
\end{theorem}
\begin{proof}
	Suppose first that $G \subseteq \po(U)$ is a generic filter and $x^G = \bigcup\{ s^p \mid p \in G\}$. To show $\vec{x}$ satisfies the alternative Mathias property it suffices to check that for every  Global Predicate System $\Phi$ and a condition $p = (s,T)$ the direct extension $p^* = (s,T^*)$ of $p$, given by $T^* = T \cap T^s$, satisfies
\[ p^* \Vdash \exists k\ \name{x}^G \setminus k \text{ fits } \Phi(\name{x}\uhr k).
\]
Indeed, let $q = (r \fr t, R)$ be an extension of $p^*$. Without loss of generality we may assume $t \in T^*$ has length $|t| \geq n^s$.
We may therefore pick a $U$-large subtree $S$ of $R \cap T^s_t$ such that $P^s(t,S)$ holds. It is then clear that $(r \fr t,S)$ is a direct extension of $q$ which forces 
\[  \name{x}^G \setminus |s\fr t| \text{ fits } \Phi(\name{x}\uhr |s\fr t|).
\]

  Next, we show that a sequence meeting this criteria will generate a generic.
    Suppose $\vec{x} = \langle x_i \mid i<\omega\rangle \in X^\omega$ satisfies the conditions of the theorem. 

    Let $D$ be a dense open set. For each condition $p \in G(\vec{x})$, applying the Weak Strong-Prikry Property, we obtain a canonical direct extension $q \leq^* p$ and $n<\omega$ such that for every $t\in T^q$ of length $|t| \geq n$, there is a subtree $S \subseteq T^q_t$ such that $(s^{p}\fr t, S)$ is a direct extension of $q\fr t$ and belongs to $D$.

    For every finite sequence $r$, let $p_r = (r, X^\omega)$ be the weakest condition containing $r$ as a stem. Define the GPS $\Phi$ by setting $\Phi(r)$ to be the $U$-SPS $(T^{q},n, P)$ in the following way (noting that $q, T^q,$ and $n$ are determined from $r$ canonically): $P(t, S)$ holds if $S$ is a subtree of $T^{q}_t$ such that $(s^{q}\fr t, S)$ is a direct extension of $q\fr t$ and belongs to $D$. This is nonempty by the weak Strong Prikry Property.

    By assumption, there is some $l$ such that $\vec{x}\setminus l$ fits $\Phi(\vec{x}\uhr l)$. Let $p = (\vec{x} \uhr l, X^\omega)$. Let $q$, $n$, and $T^q$ be as in the conclusion of the weak Strong Prikry Property with respect to $p$. Since $\vec{x}$ fits $\Phi(\vec{x}\uhr l) = (T^q,n,P)$, $\vec{x}\setminus l$ is a branch through $T^q$, and there is some $k\geq n$ and $S$ such that $P(\vec{x}\uhr (l,l+k]), S)$ holds and $\vec{x}\setminus l+k$ is a branch through $S$. In particular, this means that $(\vec{x}\uhr l+k,S)$ is a direct extension of $p_r \fr \langle x_i \mid l<i\leq k\rangle$ that is in $D$. This condition must be in $G(\vec{x})$ as well. We conclude that $G(\vec{x})$ is generic.
\end{proof}

\section{More about Tensors}\label{Section:Tensors}


Given ultrafilters $U_1,U_2$ on $X_1,X_2$ respectively, the standard coordinate production maps $\pi_i : X_1 \times X_2 \to X_i$ are Rudin-Kiesler projections. In the ZFC setting this is known to imply that the tree Prikry forcing $\po(U_1 \ltimes U_2)$ projects to both $\po(U_1)$ and $\po(U_2)$.
We prove the same holds in our choiceless setting. We start by introducing a number of useful notations and conventions.

\begin{definition}(Pairing sequences)
\begin{itemize}
    \item For two sequence $s_1 = \la x_i\ra_{i < len(s_1)}$, $s_2 =\la y_i\ra_{i < len(s_2)}$ of the same length $len(s_1) = len(s_2)$  define
    \[s_1 \ptimes s_2 = \la (x_i,y_i) \ra_{i < len(s_1)}.
    \]


    \item Given sets $T_1 \subseteq X_1^{<\omega}$, $T_2 \subseteq X_2^{<\omega}$ define 
    \[
    T_1 \ptimes T_2 = \{ s_1 \ltimes s_2 \mid 
    s_1 \in T_1,\ s_2 \in T_2, \ len(s_1) = len(s_2)\}.
    \]
\end{itemize}    
\end{definition}

\begin{remark}
    Suppose $U_1,U_2$ are ultrafilters on sets $X_1,X_2$, respectively, $0< n < \omega$, and $Z_i \in U_i^n$ for $i = 1,2$. Then it is clear from the definition that $Z_1 \ptimes Z_2$ belongs to $U_1 \ltimes U_2$. Hence, if $T_1 \subseteq X_1^{<\omega}$ $T_2 \subseteq X_2^{<\omega}$ are trees so that each $T_i$ is $U_i$-large, then $T_1 \ptimes T_2$ is a $U_1 \ltimes U_2$-large tree.
\end{remark}

\begin{definition}\label{Def:Lifts}(Lifts in $\po(U_1 \ltimes U_2)$)\\
Let $U_1,U_2$ be ultrafilters on sets $X_1,X_2$, respectively.\\
Let $q_1 = (s_1,T_1) \in \po(U_1)$, $q_2 = (s_2,T_2) \in \po(U_2)$, and  $p = (t,T) \in \po(U_1 \ltimes U_2)$. 
\begin{enumerate}
    \item We say that $p$ is a \emph{simple (left) lift} of $q_1$ if $t = s_1 \ptimes t_2$ for some $t_2 \in X_2^{len(s_1)}$, and $T = T_1 \ptimes S_2$ for some $U_2$-large tree $S_2 \subseteq X_2^{<\omega}$.

    \item We say that $p$ is a \emph{simple (right) lift} of $q_2$ if  $t = t_1 \ptimes s_2$ for some $t_1 \in X_2^{len(s_2)}$, and $T = S_1 \ptimes T_2$ for some $U_2$-large tree $S_1 \subseteq X_1^{<\omega}$.

    \item For either $i = 1,2$, we say that $p$ is a (left/right) \emph{lift} of $q_i$ if there is a simple lift $q \in \po(U_1 \ltimes U_2)$ of $q_i$ such that $p$ extends $q$.
\end{enumerate}

\end{definition}

We will use the notion of a ``lift'' to prove that a Prikry generic filter for $\po(U_1 \ltimes U_2)$ naturally gives rise to generics for $\po(U_1)$ and $\po(U_2)$.

\begin{theorem}\label{THM:GenericLifts}
    Let $U_1,U_2$ be $\sigma$-complete ultrafilters on sets $X_1,X_2$, respectively. Suppose that $G \subseteq \po(U_1 \ltimes U_2)$ is a generic filter, then for both $i = 1,2$, the set
    \[G_i = \{ p' \in \po(U_i) \mid \exists p \in G, p \text{ is a lift of }p'\}\]
    is a generic filter for $\po(U_i)$.
\end{theorem}

To prove Theorem \ref{THM:GenericLifts} we first extend the notion of a lift to include lifts of the Subtree Predicate Systems and Global Predicate Systems from Section \ref{Section:Mathias}.
The following is an immediate consequence of the relevant definitions. 
\begin{lemma}\label{Lem:GPSLift}(Extending Lifts)\\
Let $U_1,U_2$ be ultrafilters on sets $X_1,X_2$, respectively.
\begin{enumerate}

\item For a $U_1$-SPS $(T_1,n,P_1)$, let $(T,n,P)$ be given by 
\( T = T_1 \ptimes X_2^{<\omega} \)
and 
\( P\left( s_1 \ptimes s_2, \ S_1 \ptimes T_2 \right) \)  for every $s_1 \in X_1^{<\omega}$ and $S_1 \subseteq X_1^{<\omega}$  with  $P_1(s_1,S_1)$, and every $s_2 \in X_2^{len(s_1)}$ and a $U_2$-large tree $T_2 \subseteq X_2^{<\omega}$.
$(T,n,P)$ is a  $U_1 \ltimes U_2$-SPS, which we refer to as the left lift of $(T_1,n,P_1)$ with $U_2$.

\item For a $U_2$-SPS $(T_2,n,P_2)$, let $(T,n,P)$ be given by 
\( T = X_1^{<\omega} \ptimes T_2 \)
and 
\( P\left( s_1 \ptimes s_2, \ X_1^{<\omega} \ptimes  S_2 \right) \) for every $s_2 \in X_2^{<\omega}$ and $S_2 \subseteq X_2^{<\omega}$  with  $P_2(s_2,S_2)$, and every $s_1 \in X_1^{len(s_2)}$ and a $U_1$-large tree $T_1 \subseteq X_1^{<\omega}$.
$(T,n,P)$ is a  $U_1 \ltimes U_2$-SPS, which we refer to as the right lift of $(T_2,n,P_2)$ with $U_1$.

\item 
Given a GPS for $U_1$, $\Phi_1$ define its right lift with $U_2$ be the the the function $\Phi$ such that for each $\ell < \omega$ and 
$s_1 \ptimes s_2 \in (X_1 \times X_2)^{\ell}$, $\Phi(s_1 \ptimes s_2)$ is the right lift with $U_2$ of the $U_1$-SPS $\Phi_1(s_1)$.
$\Phi$ is a GPS for $U_1 \ltimes U_2$.

\item 
Given a GPS for $U_2$, $\Phi_2$ define its left lift with $U_1$ be the the the function $\Phi$ such that for each $\ell < \omega$ and 
$ s_1 \ptimes s_2 \in (X_1 \times X_2)^{\ell}$, $\Phi(s_1 \ptimes s_2)$ is the right lift with $U_1$ of the $U_2$-SPS $\Phi_2(s_2)$.
$\Phi$ is a GPS for $U_1 \ltimes U_2$.
\end{enumerate}
\end{lemma}

\begin{proof}(Theorem \ref{THM:GenericLifts})\\
Let $\vec{x}  = \bigcup \{ s^p \mid p \in G\}$ be the $\po(U_1\ltimes U_2)$-Prikry generic sequence associated with $G$. 
Write $\vec{x} = \vec{x}^1 \ptimes \vec{x}^2$ where for each $i = 1,2$, $\vec{x}^i \in X_i^\omega$.
We start by verifying that for each $i = 1,2$, $G_i = G(\vec{x}^i)$ is the filter derived from $\vec{x}^i$. We do this for $i = 1$ as the case $i = 2$ is essentially the same. 
Suppose that $p' = (\vec{x}^1\uhr k,T_1) \in G(\vec{x}^i)$.
Then $p = (\vec{x}^1\uhr k \ptimes \vec{x}^2\uhr k, T_1\ptimes X_2^\omega)$ is a left lift of $p'$ and clearly belongs to $G(\vec{x})= G$. Hence $p' \in G_1$. 
Suppose now that $p' = (s_1,T_1) \in G_1$ and let $p = (s,T) \in G$ be a lift of $p'$. We may assume that $p$ is a simple lift of $p'$ and therefore $s = s_1 \ptimes s_2$, $T = T_1 \ptimes T_2$. Let $k = len(s_1)$. Having $p \in G = G(\vec{x})$ we conclude that $s = \vec{x}\uhr k$ and $\vec{x}\setminus k \in [T]$. The last implies 
$s_1 = \vec{x}^1\uhr k$ and $\vec{x}^1\setminus k \in [T_1]$, witnessing $p' \in G(\vec{x}^1)$. 

Having verified $G_i = G(\vec{x}^i)$ is a filter, it remains to show it is generic, or equivalently, 
that $\vec{x}^i$ is a Prikry-generic sequence. 
For this, we verify that $\vec{x}^i$ satisfies the choiceless Mathias Criterion (Theorem \ref{THM:AlternativeMathias}).
Let $\Phi_i$ be a GPS for $U_i$, and $\Phi$ be its lift to $U_1 \ltimes U_2$ from Lemma \ref{Lem:GPSLift}.
As $\vec{x} = \vec{x}^1 \ptimes \vec{x}^2$ is $\po(U_1 \ltimes U_2)$-generic, there is $m < \omega$ such that
$\vec{x}\setminus m$ fits the $U_1 \ltimes U_2$-SPS $\Phi(\vec{x}\uhr m)$. 
It now follows from the definition of the lift $\Phi$ that $\vec{x}^i\setminus m$ fits the $U_1$-SPS $\Phi_i(\vec{x}^i\uhr m)$. 
\end{proof}

\subsection{When $\po(U_1 \ltimes U_2)$ is equivalent to $\po(U_1) \times \po(U_2)$}



The following is shown in \cite{goldberg:measurablecardinals}.
\begin{lemma}${}$
\begin{enumerate}
    \item If $\kappa$ is almost super compact then $\aleph(\power(Z)) < \kappa$ for every $Z \in V_\kappa$.
    
    \item If $U$ is a $\kappa$-complete ultrafilter then $U$ is $Y$-closed for every set $Y$ with $\aleph(Y) < \kappa$. 
\end{enumerate}
\end{lemma}

 The following is an extension of a ZFC result concerning the connection between tensor products and standard products of ultrafilters.  Blass \cite{BlassPhD} studied products of ultrafilters. The relation between products and tensors was recently studied by Goldberg in \cite{GoldbergProduct}, where the following is observed.

\begin{lemma}(\cite{GoldbergProduct})\label{Lem:ProdInsideTensor}
Suppose that $U_1,U_2$ are ultrafilters on sets $X_1,X_2$, respectively. If $U_2$ is $X_1$-closed then for every set $A \in U_1 \ltimes U_2$ there are sets $B_1\in U_1,B_2 \in U_2$, which are definable from $A,X_1,U_2$, such that $B_1 \times B_2 = \{ (x_1,x_2) \mid x_1 \in B_1, x_2 \in B_2\}$ is contained in $A$. 
\end{lemma}

It follows from the assumption of the Lemma that the product filter $U_1 \times U_2$, generated by the set of cartesian products $B_1 \times B_2$, $B_i \in U_i$, is the ultrafilter $U_1 \ltimes U_2$. In particular, the Prikry forcing $\po(U_1 \ltimes U_2) = \po(U_1 \times U_2)$. 
It is therefore natural to compare $\po(U_1 \ltimes U_2)$ with the product forcing $\po(U_1) \times \po(U_2)$. 

\begin{lemma}\label{Lem:TensorProd2Prod}
    Suppose that $U_1,U_2$ are $\sigma$-complete ultrafilters on sets $X_1,X_2$, respectively. If $U_2$ is $\power(X_1)$-closed then the forcing posets $\po(U_1 \ltimes U_2)$ and $\po(U_1) \times \po(U_2)$ are equivalent. 
\end{lemma}
\begin{proof}
    We verify there is a dense embedding $i : \po(U_1) \times \po(U_2) \to \po(U_1\ltimes U_2)$. We take $\dom(i)$ to be the dense subset of
    $\po(U_1) \times \po(U_2) $ consisting of all pairs $\la p^1,p^2\ra = \la (s^1,T^1), (s^2,T^2)\ra$ with $len(s^1) = len(s^2)$, and define
    
    \[i(\la p^1,p^2\ra) = (s^1 \ltimes s^2, T^1 \ltimes T^2) \in \po(U_1 \ltimes U_2)\]

    It is clear that $i$ is an order preserving and injective. It therefore remains to show that the image of $i$ is dense in $\po(U_1 \ltimes U_2)$. \\
    To this end, let $(s, T) \in \po(U_1 \ltimes U_2)$. 
    Clearly, $s = s^1 \ltimes s^2$ for some $s^i \in [X_i]^{len(s)}$, $i = 1,2$.
    We need to find trees $T^i$, $i = 1,2$, that are $U_i$-large, respectively, such that $T^1 \ltimes T^2 \subseteq T$. 
    For each sequence $t = t_1 \ltimes t_2 \in T$, denote the sets resulting from Lemma \ref{Lem:ProdInsideTensor} with respect to the successor set $\suc_T(t) \in U_1 \ltimes U_2$ by $B_i(t) \in U_i$, $i = 1,2$. By the same lemma, $B_1(t),B_2(t)$ are definable from $T,t,U_1,U_2$.
    We use these sets to define sets $B_{1,n} \in U_1^n$ and $B_{2,n}\in U_2^n$ for each $n$, satisfying 
\begin{itemize}
    \item $B_{1,n} \times B_{2,n} \subseteq T \cap (X_1 \times X_2)^n$.

    \item for every $i =1,2$ and $m < n$, $B_{i,n} \cap X_i^m \subseteq B_{i,m}$.
\end{itemize}

Starting from $n = 1$, define $B_{i,1} = B_i(\la \ra)$ for $i = 1,2$. Then $B_{1,1} \times B_{2,1} \subseteq \suc_T(\la \ra) = T \cap (X_1 \times X_2)$.\\ 
Suppose that $B_{1,n},B_{2,n}$ have been defined for some $n \geq 1$ and satisfy the conditions listed above. 
For each $t_2 \in B_{2,n}$, the sets
$B_2({t_1 \ltimes t_2})$, $t_1 \in B_{1,n}$, are all members of $U_2$. Since $U_2$ is $X_1$-closed, the intersection $B_{2,t_2} = \bigcap_{t_1 \in B_{1,n}} B_2(t_1 \ltimes t_2)$ belongs to $U_2$.

Next, we have a definable assignment, mapping each $t_2 \in B_{2,n}$ to a function $t_1 \mapsto B_{1}(t_1 \ltimes t_2)$, from $B_{1,n}$ to $U_1 \subseteq \power(X_1)$.
Since $U_2$ is $\power(X_1)$-closed then so is $U_2^n$. Hence, there are $B_{2,n}^* \subseteq B_{2,n}$, $B^*_{2,n} \in U_2^n$, and a map $t_1 \mapsto B_{1,t_1}$ from $B_{1,n}$ to $U_1$ such that 
$B_{1,t_1} = B_{1}(t_1 \ltimes t_2)$ for every $t_2 \in B_{2,n}^*$ and $t_1 \in B_{1,n}$. 
We conclude that the sets
\[B_{1,n+1} = \{ t_1 \fr \la x_1\ra \mid t_1 \in B_{1,n}, x_1 \in B_{1,t_1}\}\]

and 
\[
B_{2,n+1} = \{ t_2 \fr \la x_2\ra \mid t_2 \in B_{2,n}^*, x_2 \in B_{2,t_2}\}
\]
belong to $U_1^{n+1},U_2^{n+1}$, respectively, and satisfy the inductive conditions listed above.  \\

\noindent
This conclude the construction of the set $B_{1,n},B_{2,n}$, $n < \omega$. 
We finally define trees
$T^1 \subseteq X_1^{<\omega}$, $T^2 \subseteq X_2^{<\omega}$ by 
\[T^i = \{ t \in X_i^{<\omega} \mid \forall n \geq len(t) \  \exists t'  \in B_{i,n} \  t = t'\uhr len(t) \}.\]
Since $U_1,U_2$ are $\sigma$-closed, each $T^i$, $i = 1,2$ is a $U_i$-large tree.
Our listed properties of the sets $B_{1,n},B_{2,n}$, $n < \omega$, guarantee that $T^1 \ltimes T^2 \subseteq T$. 
\end{proof}

\begin{lemma}
    Suppose that $U_1,U_2$ are $\sigma$-complete ultrafilters on infinite sets $X_1,X_2$, respectively, and $Z$ is a set. 
    If $U_2$ is $\power(X_1) \times Z$-closed then for every generic filter $G \subseteq \po(U_1 \ltimes U_2)$-generic filter, all subsets of $Z$ in $V[G]$ belong to the intermediate generic extension $V[G_1]$ by the $G$-induced generic filter $G_1 \subseteq \po(U_1)$ 
    from Theorem \ref{THM:GenericLifts}.
\end{lemma}
\begin{proof}
    By Lemma \ref{Lem:TensorProd2Prod}, $\po(U_1 \ltimes U_2)$ is equivalent to the product forcing $\po(U_1) \times \po(U_2)$. It remains to verify that every $\po(U_1) \times \po(U_2)$-name $\tau$ for a subset of $Z$ is equivalent to a $\po(U_1)$-name. Having a product forcing, we can identify $\tau$ with a $\po(U_2)$-name for a $\po(U_1)$-name for a subset of $Z$. As every $\po(U_1)$-name for a subset of $Z$ is equivalent to a name contained in $\po(U_1) \times \{ \check{z} \mid z \in Z\}$, which has cardinality $\power(X_1) \times Z$,  and $U_2$ is $\power(X_1) \times Z$-closed, it follows that $\po(U_2)$ does not introduce new $\po(U_1)$-names for subsets of $Z$. 
\end{proof}

The following Corollary is an immediate consequence of the last Lemmas.
\begin{corollary}\label{COR:PrikrySplitAtStrongLimit}
    Suppose that $U_1,U_2$ are $\sigma$-complete ultrafilters and $\kappa$ is a strong limit cardinal such that $U_1 \in V_\kappa$ and $U_2$ is $\kappa$-complete.
    Then $\po(U_1\ltimes U_2)$ is equivalent to the product forcing $\po(U_1) \times \po(U_2)$, and for every $\gamma < \kappa$, every $\po(U_1 \ltimes U_2)$-name for subsets of $V_\gamma$ every $\gamma < \kappa$ is equivalent to a $\po(U_1)$-name. 
\end{corollary}

\section{Tensor Prikry Systems}\label{sec:tensorsystems}

Throughout this section we assume there is a rank Berkeley cardinal $\lambda$ in $V$, $\kappa_0 \geq \lambda$ is an almost supercompact cardinal, and $\eta > \kappa_0,\lambda$ is a strongly inaccessible cardinal. 
Since 
\[V_\eta \models \lambda \text{ is rank Berkeley}\]
By \cite{goldberg:measurablecardinals}
$V_\eta \models \text{ there is a proper class of almost supercompacts and a proper class of regular cardinals.}$
Let $\la \kappa_i \mid i < \eta\ra$ be an increasing enumeration of all almost supercompact cardinals in $V_\eta$ starting from $\kappa_0$. 
Set 
\[\Delta = \{ \delta < \eta \mid \kappa_0 \leq \delta \text{ is a regular cardinal}\}\]

Define for each $\delta \in \Delta$ $$i(\delta) = \max( \{i \in Ord \mid \delta \geq \kappa_i\} ),
$$
\[\mathcal{U}_\delta = \{ U \subseteq \power(\delta) \mid U \text{ is a } \kappa_{i(\delta)}\text{-complete uniform ultrafilter on }\delta \}\] 

By Goldberg \cite[Theorem 3.3]{goldberg:measurablecardinals}, each $\mathcal{U}_\delta$ is nonempty, and by Theorem \ref{Fact:AlmostSCWellOrderable}, the set $\bigcup_{\delta \in \Delta} \mathcal{U}_\delta$ is well-orderable. 
We therefore fix a choice function 
\[\vec{U} = \la U_\delta \mid \delta \in \Delta\ra \in \prod_{\delta \in \Delta}\mathcal{U}_\delta.\]

\begin{definition}
For a finite nonempty set $d \subseteq \Delta$, whose increasing enumeration is $d = \la \delta_n \mid n < k\ra$, define the associated ultrafilter
 $U_d$ to be the
left tensor product
\[
U_d := U_{\delta_1} \ltimes U_{\delta_2} \ltimes \dots \ltimes U_{\delta_{k-1}}.
\]
\end{definition}

\noindent

\begin{remark}\label{RMK:TensorBasics}${}$
\begin{enumerate}

    \item The notion of a lift in Definition \ref{Def:Lifts} of a condition $p$ in $\po(U_1)$ or $\po(U_2)$ to a condition $p' \in \po(U_1 \ltimes U_2)$, naturally extends to lifts from a condition $p$ in $\po(U_d)$ for a finite tensor given by 
    $d \in [\Delta]^{<\omega}$ to a condition $p'$ in a bigger tensor $\po(U_{d'})$ for $d' \supseteq d$. Theorem \ref{THM:GenericLifts} clearly extends to this setting, to show that given a generic filter $G_{d'} \subseteq \po(U_{d'})$, its induced filter
    \[
    G_d = \{ p \in \po(U_d) \mid \exists p'\in G_{d'}, p' \text{ is a lift of } p\}
    \]
    is generic for $\po(U_d)$.

\item For each $i \in Ord$, since $\kappa_i$ is a strong limit cardinal, by Lemma \ref{Lem:TensorProd2Prod} it follows that for any $d \in [\Delta]^{<\omega}$
\[U_d =  U_{d\cap \kappa_i} \times U_{d\setminus \kappa_i}.\]

Moreover, by Lemma \ref{Lem:TensorProd2Prod},
\[\po(U_d) \equiv \po(U_{d\cap \kappa_i}) \times \po(U_{d\setminus \kappa_i}).\]

And by Corollary, \ref{COR:PrikrySplitAtStrongLimit}, all bounded subsets of $V_{\kappa_i}$ added by $\po(U_d)$ belong to the intermediate generic extension by $\po(U_{d\cap \kappa_i})$.
\end{enumerate}
\end{remark}

\begin{definition}
    Define the class forcing $\po_{\fin}(\vec{U})$ by 
    setting its domain to be the disjoint union 
    \[
    \po_{\fin}(\vec{U}):= \biguplus_{d \in [\Delta]^{<\omega}} \po(U_d)
    \]

    For each $p \in \po_{\fin}(\vec{U})$ let
    $\supp(p) \in [\Delta]^{<\omega}$ be the unique so that $p \in \po(U_{\supp(p)})$.

    A condition $p' \in \po_{\fin}(\vec{U})$ extends $p$ if $\supp(p) \subseteq \supp(p')$ and $p' \in \po(U_{\supp(p')})$ is a lift of $p \in \po(U_{\supp(p)})$.
\end{definition}

\begin{lemma}\label{LEM:PfinSingularizing}
    The forcing $\po_{\fin}(\vec{U})$ adds a cofinal $\omega$ sequence $\vec{x}_\delta$ in each regular $\delta \geq \kappa_0$.
\end{lemma}
\begin{proof}
It is clear that for each $\delta \in \Delta$, the set of conditions $p \in \po(\vec{U})$ such that $\delta \in supp(p)$ is dense open. 

Therefore, 
a $V$-generic filter $G \subseteq  \po(\vec{U})$ adds a Prikry generic sequence $\vec{x}^G_\delta = \la x^G_{\delta}(n) \mid n < \omega\ra \subseteq \delta$ for the Prikry forcing $\po(U_\delta)$. Since $U_\delta$ is uniform ultrafilter on $\delta$, the $\omega$-sequence $\vec{x}^G_\delta$ is cofinal in $\delta$.
\end{proof}

Next, we will need the fact that the local factoring to products of $\po(U_{w})$ pointed out in Remark \ref{RMK:TensorBasics} extends to initial segments of $\po_{\fin}(\vec{U})$. 
\begin{lemma}\label{LEM:System2Product}
    Suppose that $\rho \leq \kappa$ where  $\kappa = \kappa_i$ is an almost supercompact cardinal, and 
    $d\subseteq \Delta \setminus \kappa$ is a finite set. 
    Then the forcing 
    $\po_{\fin}(\vec{U}\uhr (\rho \cup d))$
    is equivalent to the product
    \[\po_{\fin}(\vec{U}\uhr \rho) \times \po_{\fin}(\vec{U}\uhr d)\]
\end{lemma}
\begin{proof}
    Let $D \subseteq \po_{\fin}(\vec{U}\uhr (\rho \cup d))$ be the set of conditions $p$ such that 
    \begin{itemize}
        \item $supp(p)$ contains $d$
        \item $(s^p,T^p) \in \po(U_{\supp(p)})$ is of the form  $s^p = s_1 \ptimes s_2$ and $T^p = T_1 \ptimes T_2$ where
        $(s_1,T_1) \in \po(U_{\supp(p)\uhr \rho}) \in \po_{\fin}(\vec{\mathcal{U}}\uhr \rho)$, and 
        $(s_2,T_2) \in \po(U_d) \subseteq \po_{\fin}(\vec{U}\uhr d)$.
    \end{itemize}
    By Remark \ref{RMK:TensorBasics}, $D$ is dense. And clearly, the map $p \mapsto (p_1,p_2)$ where $p_i = (s_i,T_i)$, is an order preserving projection from $D$ onto a dense open set in  $\po_{\fin}(\vec{U}\uhr \rho) \times \po_{\fin}(\vec{U}\uhr d)$.
\end{proof}

Given a $V$-generic class $G \subseteq \po_{\fin}(\vec{U})$,
our (nearly) final model is the intermediate symmetric extension $\mathcal{N}$ of $V$. That is, roughly, the minimal ZF extension of $V$ that contains each Prikry sequence $\vec{x}^G_\delta$, $\delta \in \Delta$. There are several ways of making the definition of $\mathcal{N}$ precise. Following Apter's approach from \cite{apter:ADsingularcardinals},  we consider the sub-language
$\mathcal{L}^1$ of the forcing language
$\mathcal{L}_{\po_{\fin}(\vec{U})}$, which extends the language of set theory, and further includes 
\begin{enumerate}
    \item constants $\check{x}$ for canonical names of ground model sets $x \in V$, 
    \item A constant $\vec{x}^G_\delta$, for the standard name of a Prikry generic filter at $\delta$, for every $\delta \in \Delta$, and 
    \item A unary predicate $\dot{V}()$ for detecting ground model sets.
\end{enumerate}
We define the ordinal rank of each constant symbol to be the name rank of its associated $\po_{\fin}(\vec{U})$-name.
A term in $\mathcal{L}^1$ is a formula $\tau = \tau(x,u_1,\dots,u_k)$ (with free variables $x,u_1\dots,u_k$) in $\mathcal{L}^1$. We define the rank of $\tau$ to be the maximum of the rank of the constant symbols appearing in it. \\

\noindent
Given a generic class $G \subseteq \po_{\fin}(\vec{U})$, we define the interpretation of $\tau_G$ of $\tau$ to be the formula resulting from $\tau(x,u_1,\dots,u_k)$ by replacing each constant symbol in it with the generic interpretation by $G$ of its associated name. 
With these, we define a hierarchy $\mathcal{N}_\alpha$, $\alpha < \eta$ by 
\begin{itemize}
    \item $\mathcal{N}_0 = \emptyset$
    \item $\mathcal{N}_{\alpha+1}$ consists of all subsets $D \subseteq \mathcal{N}_\alpha$, which are definable 
    over $\mathcal{N}_\alpha$ by a formula $\varphi(x) = \tau_G(x;a_1,\dots,a_k)$, where $\tau(x,u_1,\dots,u_k)$ is an $\mathcal{L}^1$-term of rank $\leq \alpha$, and $a_1,\dots, a_k \in \mathcal{N}_\alpha$.
    \item $\mathcal{N}_\delta = \bigcup_{\alpha < \delta}\mathcal{N}_\alpha$ for a limit ordinal $\delta$.
\end{itemize}

Let $\mathcal{N} = \bigcup_{\alpha \in \eta}\mathcal{N}_\alpha$.

\begin{definition}
    Define for each $z \in \mathcal{N}$ a support $\Delta(z) \in [\Delta]^{<\omega}$ by induction on its $\mathcal{N}$-rank. 
    Assuming $z \in \mathcal{N}_{\alpha+1} \setminus \mathcal{N}_\alpha$ and supports $\Delta(a)$ were defined for each $a \in \mathcal{N}_\alpha$, define the support $\Delta(z)$ to be 
    be the lexicographically least finite set of ordinals of the form 
     \[
    \left(\bigcup_{1 \leq i \leq k} \Delta(a_i)\right) \cup \{ \delta \in \Delta \mid \text{ the constant associated with }\vec{x}^G_\delta \text{ appears in }\tau\}
    \]
    such that there is an $\mathcal{L}^1$-term $\tau(x,u_1,\dots,u_k)$ and $a_1,\dots,a_k \in \mathcal{N}_\alpha$ such that
     $\varphi(x) = \tau_G(x,a_1,\dots,a_k)$ defines $z$ in $\mathcal{N}_\alpha$.
\end{definition}

Recall from Remark \ref{RMK:Homog} that tree Prikry forcings are homogeneous. The cone isomorphisms witnessing this are formed by switching the stem of a condition and preserving the tree.

\begin{definition}\label{Def:Prikry-Cone-Isom}
	Let $d$ be a finite subset of $\Delta$. Let $\pi$ be a cone isomorphism in $\po(U_d)$. We say that a cone isomorphism $\pi$ \emph{fixes} $\delta \in \Delta$ if for all conditions $p$ in $\dom(\pi)$ with stem $s^p = \langle s^p_\alpha \mid \alpha \in d\rangle$, we have $s^p_\delta = s^{\pi(p)}_\delta$.
Informally, $\pi$ fixes $\delta$ if it doesn't change stems that correspond to the Prikry sequence associated to $\delta$.

Given a cone isomorphism $\pi_0$ on $\po(U_d)$, we can extend it to a cone isomorphism $\pi$ on $\po(U_{d\cup \{\delta\}})$ by mapping $s^p = \langle s^p_\alpha \mid \alpha \in d\cup\{\delta\}\rangle$ to $\langle s^{\pi_0(p)}_\alpha \mid \alpha \in d\rangle \fr \langle s^p_\delta\rangle$. We call $\pi$ the \emph{trivial extension} of $\pi_0$.

Iterating this process, we can extend any cone isomorphism $\pi_0$ on $\po(U_d)$ to a cone isomorphism $\pi$ on $\po_{\fin}(\vec{U})$. We say that a cone isomorphism $\pi$ is \emph{finitely supported} if it is a trivial extension of a cone isomorphism on $\po(U_d)$ for some finite $d\subseteq \Delta$.
\end{definition}

\begin{lemma}
	For all $z\in \mathcal{N}$, there is a $\po_{\fin}(\vec{U})$-name $\dot{z}$ for $z$, such that any finitely supported cone isomorphism that fixes $\Delta(z)$ will preserve $\dot{z}$. That is, $\pi_*(\dot{z}) = \dot{z}$.
\end{lemma}
\begin{proof}
	Induction on the rank of $z$. We have canonical names for $\mathcal{N}_\alpha$, for each $G_\delta$ with $\delta \in \Delta(z)$, and (by induction) for each $a_n$ used to define $z$. Use this to build a name for $z$. The cone isomorphisms will preserve all of the pieces, so it will preserve $z$ as well.
\end{proof}

\begin{lemma}\label{LEM:canonical names}
	Let $z \in \mathcal{N}$. Then there is a name $\dot{z}$ for $z$ such that for all finitely supported cone isomorphisms $\pi$ fixing $\Delta(z)$, $\pi_*(\dot{z}) = \dot{z}$.
\end{lemma}
\begin{proof}
	We prove this by induction on the $\mathcal{N}-$rank of $z$. If $z$ is in $\mathcal{N}_0$, then $z = \varnothing$, so the emptyset is a name for $z$ satisfying the criteria.
	For each $\alpha \leq \eta$ let $\dot{\mathcal{N}}_\alpha$ be the associated $\po_{\fin}(\vec{U})$-name for $\mathcal{N}_\alpha$ given by its definition. 
	Suppose that $z \in \mathcal{N}_{\alpha+1}$. Then there is a formula $\phi(x) = \tau_G(x;a_1,\dots, a_k)$ of rank $\leq \alpha$ defining $z$, with each parameter $a_i \in \mathcal{N}_\alpha$. By definition, $\Delta(\alpha_i) \subseteq \Delta(z)$ for all $i \leq k$. So by induction, there are names $\dot{a}_i$ such that for all finitely supported cone isomorphisms $\pi$, $\pi_*(\dot{a}_i) = \dot{a}_i$.
	
	We define $\dot{z}$ to be a name for the set defined by $\tau_G(x;\dot{a}_1, \dots, \dot{a}_k)$ over $\mathcal{N}_\alpha$.
	
	Formally:
	\[\dot{z} = \{(p,\sigma) \mid p \forces \dot{\mathcal{N}_\alpha} \models \tau(\sigma, \dot{a}_1, \dots, \dot{a}_k)\}.\]
	
	Let $\pi$ be a cone isomorphism that fixes $\Delta(z)$. By induction, $\pi_*(\dot{a}_i) = \dot{a}_i$.
	
	Recall that $\tau$ is a formula in the language $\mathcal{L}^1$, which means it can refer to canonical names for ground model sets, standard names for Prikry generics $\vec{x}_\delta^G$, and a unary predicate $\dot{V}$ that detects ground model sets. By definition, $\tau$ will only refer to $\vec{x}_\delta^G$ for $\delta \in \Delta(z)$, and the interpretations of those names won't be changed by $\pi$. Similarly, the interpretation of $\dot{\mathcal{N}}_\alpha$ will also be fixed by $\pi$.
	
	Since $\pi$ fixes each $\dot{a_i}$ and won't change the truth of $\tau$, we conclude that $p \forces \dot{\mathcal{N}}_\alpha\models \tau(\sigma, \dot{a}_1, \dots, \dot{a}_k)$ if and only if $\pi(p) \forces \dot{\mathcal{N}}_\alpha\models \tau(\sigma, \dot{a}_1, \dots, \dot{a}_k)$. We conclude that $\dot{z} = \pi_*(\dot{z})$.
\end{proof}

We can use the above observations to generalize standard homogeneity arguments (e.g., as in the proof of Lemma 2.1 in \cite{apter:ADsingularcardinals}), and obtain a connection between statements about sets $z \in \mathcal{N}$ and their supports $\Delta(z)$.

\begin{lemma}\label{LEM:StandardArguments}${}$
	\begin{enumerate}
		\item For every set $z \in \mathcal{N}$, let $\dot{z}$ be the corresponding name constructed in Lemma \ref{LEM:canonical names}. For every formula $\psi(x)$ in the language of set theory, the truth value of $\psi(\dot{z})$ in $\mathcal{N}$ is determined by the restriction of $G$ to the Prikry sequences in $\Delta(z)$, $G\uhr \Delta(z) = \langle \vec{x}^G_\delta \mid \delta \in \Delta(z)\rangle$. 
		
		\item For every $X \in V$ and a subset $z \subseteq X$, $z \in \mathcal{N}$, $z \in V[\langle \vec{x}^G_\delta \mid \delta \in \Delta(z)\rangle]$
	\end{enumerate}
\end{lemma}
\begin{proof}
	\begin{enumerate}
		\item Let $\psi(x)$ be a formula in the language of set theory, and let $z\in \mathcal{N}$. Let $\dot{z}$ be the name for $z$ given by Lemma \ref{LEM:canonical names}. 
        Let $p$ and $r$ be conditions in $\po_{\fin}(\vec{U})$ such that $p \forces {\dot{\mathcal{N}}} \models \psi(\dot{z})$, $r\forces {\dot{\mathcal{N}}} \models \lnot \psi(\dot{z})$. We claim that $p,r$ must give contradictory information about the Prikry sequences $\vec{x}_\delta^G$ for some $\delta \in \Delta(z)$.
        Suppose this is not the case. By extending $p,r$ if needed, we may assume that $\supp(p) = \supp(r) \supseteq \Delta(z)$ and that for each $\delta \in \Delta(z)$, 
        the $\delta$-th coordinate of the stem of $p$ is equal to the $\delta$-th coordinate in the stem of $r$. Moreover, by intersecting trees if needed, we can assume that they have the same trees $T^p = T^r$.
        
		
		Using the homogeneity of each coordinate of $\po_{\fin}(\vec{U})$ as described in Remark \ref{RMK:Homog}, define a cone isomorphism $\pi_0$ on $\po(U_{\supp(p)})$ between the cones above $p$ and $r$ that fixes $\Delta(z)$. 
        Let $\pi$ be the trivial extension of $\pi_0$ to $\po_{\fin}(\vec{U})$ given in Definition \ref{Def:Prikry-Cone-Isom}. Then $\pi$ is a finitely supported cone isomorphism fixing $\Delta(z)$, such that $\pi(p) = r$. By Lemma \ref{LEM:canonical names}, $\pi_*(\dot{z}) = \dot{z}$. We conclude that $r = \pi(p) \forces \dot{\mathcal{N}} \models \psi(\dot{z})$, contradicting the fact that $r\forces \dot{\mathcal{N}} \models \lnot \psi(\dot{z})$.
		
		\item Let $\dot{z}$ be the name for $z$ given in Lemma \ref{LEM:canonical names}. For each $x \in X$, by part 1 the truth value of $x\in \dot{z}$ in $\mathcal{N}$ can be determined in $V[G\uhr \Delta(z)]$. Thus $z \in V[G\uhr \Delta(z)]$.
		
	\end{enumerate}
\end{proof}

\begin{proposition}\label{Prop:ZFinN}
    $\mathcal{N} \models ZF$
\end{proposition}
\begin{proof}
    The verification that $\mathcal{N}$ satisfies all axioms in ZF-PowerSet is standard, building on Lemma \ref{LEM:StandardArguments}. We turn to prove $\mathcal{N} \models PowerSet$.
    To this end, we prove by induction that for every ordinal $\gamma$ there is some $\rho_\gamma$ such that every set $z \in V_{\gamma}^{\mathcal{N}}$ has a name $\name{z} \in V_{\rho_\gamma}^{\po_{\fin}(\vec{U}\uhr \rho_\gamma)}$ (I.e., names of rank $\rho_\gamma$ in the set tensor system forcing associated with the sequence $\vec{U}\uhr \rho_\gamma$) and $\Delta(z) \subseteq \Delta \cap \rho_\gamma$.
    The case $\gamma = 0$ and limit steps of the induction are trivial. Suppose therefore that the statement holds for $\gamma$ with $\rho_\gamma$ being a witness. We wish to prove it for $\gamma+1$.
    Let $\kappa= \kappa_i$ be the minimal almost supercompact cardinal above $\rho_\gamma,\gamma$. We claim that $\rho_{\gamma+1} = \kappa$ works. 
    To see this, fix some $z \in V_{\gamma+1}^{\mathcal{N}}$.
    Let $\name{z}$ be a name for $z$, and let $p \in G$ be a condition forcing that $\name{z} \subseteq V_\gamma^{\mathcal{N}}$ and that $\Delta(\name{z}) = \check{\Delta(z)}$.
    Let 
    \[\rho^* = \max( \{\rho_\gamma\} \cup (\Delta(z) \cap \kappa)) \text{ and }
    d = \Delta(z) \setminus \kappa.\] 
    It  
    follows from the inductive assumption for $\gamma$ and $\rho_\gamma$ and Lemma \ref{LEM:StandardArguments}, that for
    every $y \in z$, there are
    \begin{enumerate}
        \item a name $\name{y} \in V^{\po(\vec{U}\uhr \rho_\gamma)}_{\rho_\gamma}$, and 
        \item condition $q \in \po(\vec{U}\uhr(\rho^* \cup d))$ such that
    $q \Vdash \name{y} \in \name{z}$.
    \end{enumerate} 
    Let $\name{z}^*$ be the set of all pairs $(q,\name{y})$ of the above form. 
    Then $p \Vdash \name{z} = \name{z}^*$.
    Now, $\name{z}^*$ is a $\po(\vec{U}\uhr(\rho^* \cup d))$-name for a subset of a set of cardinality less or equal to $V^{\po(\vec{U}\uhr \rho_\gamma)}_{\rho_\gamma} \in V_\kappa$.
    By Lemma 
    \ref{LEM:System2Product}, we can identify $\name{z}^*$ with a name by  
    $\po(\vec{U}\uhr \rho^*) \times \po(\vec{U}\uhr d)$.
    Since all the ultrafilters in $\vec{U}\uhr d$ are $\kappa$-complete, it follows from Corollary \ref{COR:NoNewSmallSets} that
    $\name{z}^*$ can be identified with 
     $\po(\vec{U}\uhr \rho^*)$-name. Since $\rho^* < \kappa$, we conclude that for every $z \in V_{\gamma+1}^{\mathcal{N}}$ has a name $\name{z} \in V^{\po(\vec{U}\uhr \kappa)}_\kappa$. And it is then clear that $\Delta(z) \subseteq \kappa$. 
     Hence $\rho_{\gamma+1} =\kappa$ has the desired properties.
\end{proof}

Being a model of ZF, we know that $\mathcal{N}$ has a proper class of cardinals. Using the results from the previous sections we have the following estimation of this class.

\begin{theorem}
    Every strong limit cardinal in $V$ above $\lambda$ remains a cardinal in $\mathcal{N}$.
\end{theorem}
\begin{proof}
    Let $\rho > \lambda$ be a strong limit cardinal in $V$, and $f: \gamma \to \rho$ be a function in $\mathcal{N}$.  We want to show that $f$ is not onto $\rho$. 
    By the second clause of Lemma \ref{LEM:StandardArguments} there is a finite set $d = \Delta(f)$ such that $f \in V[G\uhr d]$. But $G_d$ is generic for the Prikry forcing $\po(U)$ where $U = U_{w(p)}$ for any $p \in G$ with $\supp(p) = d$, and by Corollary \ref{COR:PrikryNoCollapseStong}, $\po(U)$ cannot collapse strong limit cardinals above $\lambda$. 
\end{proof}

We can now prove the main result of this section.
\begin{proof}(Theorem \ref{THM:Main4})\\
Let $\lambda$ be a rank Berkeley cardinal in $V$, $G \subseteq \po_{\fin}(\vec{U})$  be a generic class, and $\mathcal{N}$ the resulting  symmetic model. 
By Proposition \ref{Prop:ZFinN} $\mathcal{N}$ is a model of ZF, and by Lemma \ref{LEM:PfinSingularizing}, every cardinal $\delta \geq \kappa_0$ is singular in $\mathcal{N}$. Let $\mathcal{N}^*$ be the generic extension of $\mathcal{N}$ obtained by collapsing some $\rho \geq \kappa_0$ to $\omega$. Then 
\[\mathcal{N}^* \models \text{ZF + Every uncountable cardinal is singular.}
\]
\end{proof}

\section{Further Applications and Discussions}
\subsection{The first measurable cardinal can be the first strongly inaccessible}
In \cite{gitik-hayut-karagila:FirstMeasurable}, Gitik, Hayut, and Karagila established the consistency of a ZF model where the first measurable cardinal is also the first strongly inaccessible. They point out that if one weakens the goal from strongly inaccessible to weakly inaccessible, then a version of Apter's forcing from \cite{apter:ADsingularcardinals} suffices. I.e., , if Prikry sequences are added to every regular below a weakly inaccessible cardinal $\delta <\Theta$, then $\delta$ remains measurable and becomes the first weakly inaccessible in the corresponding symmetric model. \\
The constructions leading to the proof of Theorem \ref{THM:Main4} above are amenable to a similar argument, allowing us to construct a model where the first strongly inaccessible cardinal is the first measurable.

\begin{proof}(Theorem \ref{THM:Main5})
    Suppose that $\eta > \lambda$ is a regular limit of almost supercompact cardinals $\la \kappa_i \mid i < \eta\rangle$. Then $\eta$ is strongly inaccessible and the club filter on $\eta$ is $\eta$-complete (see \cite{goldberg:measurablecardinals}). In particular, an ultrafilter $U_\eta$ corresponding to atoms of the club filter on $\eta$ are  $\eta$-complete and uniform. Now, repeating the construction from the previous section let $\vec{U} = \la U_\delta \mid \kappa_0 < \delta < \eta, \text{is regular}\ra$. We set $\po_{\fin}(\vec{U})$ to be the associated tensor system, which adds Prikry sequences to all regulars $\kappa_0 \leq \delta < \eta$.
    Let $G \subseteq \po_{\fin}(\vec{U})$ be $V$-generic and $\mathcal{N} =\bigcup_{\alpha \in Ord}\mathcal{N}_\alpha$ be the associated class version of the symmetric model, and $\mathcal{N}^*$ be the extension of $\mathcal{N}$ obtained by collapsing $\kappa_0$ to $\omega$. The arguments of the proof of Theorem \ref{THM:Main4} give that all uncountable cardinals below $\eta$ are singular. Lemma \ref{LEM:StandardArguments} shows that every  subset $z$ of $\eta$ belongs to an intermediate extension 
     $V' = V[\la \vec{x}^G_\delta \mid \delta \in \Delta(z)\ra]$ which is generic with respect to a single Prikry generic by a tensor $\po(U_{\Delta(z)})$, where $\Delta(z) \in [\eta]^{<\omega}$. Since $\eta$ is strong limit and measurable and $\po(U_{\Delta(z)}) \in V_\eta$, it follows from the Levy-Solovay argument (see \cite{LevySolovay}) that $U_\eta$ generates a $\eta$-complete ultrafilter in $V'$. Hence $U_\eta$ generates a $\eta$-complete uniform ultrafilter on $\eta$ in $\mathcal{N}^*$, and is therefore the first measurable cardinal.
\end{proof}
\subsection{The Continuum Problem in models $\mathcal{N}^*$}
The construction of a model $\mathcal{N}^*$ of ZF + every uncountable cardinal is singular, allows some flexibility in choosing the cardinal $\rho$ to become $\omega_1$; any choice of a cardinal $\kappa_0 \leq \rho < \eta$ could produce a version of the model $\mathcal{N}^*$ where every uncountable cardinal is singular.
One way the different resulting models could be different is in the properties of their continuum.

\begin{question}
 What can be said about the continuum in extensions of the form $\mathcal{N}^*$ of the previous section?  
\end{question}

Two factors that come into play when addressing this question are the status of the generalized continuum problem in the ground model $V$, and the new bounded subsets of $\rho$ which are added by the forcing. Regarding the former, 
Goldberg (\cite{goldberg:choicelesscardinals}) studied the influence of choiceless large cardinals on the continuum problem. In particular, he proved the following:
\begin{theorem}\cite[Theorem 8.3]{goldberg:choicelesscardinals}
    If $\lambda$ is Rank Berkeley, $\ka \geq \lambda$ is rank-reflecting, and $\epsilon \geq \ka$ is an even ordinal, then there is a surjection from $\mathcal{P}(\theta_\epsilon)$ onto $\theta_{\epsilon+1}.$
\end{theorem}
\begin{corollary}\cite[Corollary 8.4]{goldberg:choicelesscardinals}
   If $\epsilon$ is an even ordinal and $j:V_{\epsilon+1} \to V_{\epsilon+1}$ is an elementary embedding with critical point $\kappa$, then the set of regular cardinals in the interval $(\theta_\epsilon, \theta_{\epsilon+1})$ has cardinality less than $\ka$.
\end{corollary}

Regarding the new bounded subsets added to $\rho$ by the forcing, we know that if $\kappa_i \geq \rho$ is the first almost supercompact greater or equal to $\rho$ then the forcing $\po_{\fin}(\vec{U}\setminus 
\kappa_i)$ will not add new bounded subsets to $\rho$. It remains unclear what is the effect of $\po(U_\delta)$ on $\power_{bd}(\rho)$ for a regular cardinal $\rho \leq \delta < \kappa_i$. 

\begin{question}
    How many bounded subsets of $\rho$ are added by a Prikry forcing $\po(U_\delta)$?
\end{question}

\subsection{Strong measures}

One major motivation for the results in this paper is the question of whether well-founded ultrafilters on well-orderable sets above a rank Berkeley behave more like ultrafilters derived from measurable cardinals, or more like ultrafilters derived from strongly compact or supercompact cardinals. 
Our results (e.g., Theorems \ref{THM:Main1} and \ref{THM:Main2}) place a uniform bound of the strength these ultrafilters. Since these ultrafilters are natural candidates for witnessing large cardinal strength in HOD, it is natural to ask if the existence of a rank Berkeley cardinals places a bound on the large cardinals in HOD. 

\begin{question}
    Does ZF + there is a rank Berkeley cardinal imply the nonexistence of a strongly compact cardinal in HOD?
\end{question}

Many characterizations of large cardinals that are equivalent in ZFC are no longer the same in ZF, and properties of a large cardinal in the choice setting may not follow from the existence of these large cardinals without choice.
The first rank Berkeley seems to be a tipping point in the characterization of supercompactness-like properties. As shown in \cite{goldberg:choicelesscardinals}, every almost supercompact below the first rank Berkeley is supercompact or the limit of supercompacts. Above a rank Berkeley, almost supercompactness is equivalent to rank reflection, a seemingly much weaker notion. It remains unclear whether these properties can imply the existence of certain type of strong ultrafilters.

\bibliography{bib}
\bibliographystyle{plain}

\end{document}